\documentclass[11pt,reqno]{amsart}
\usepackage{fancyhdr}
\usepackage{shuffle}
\usepackage{lmodern}
\usepackage{amssymb}
\usepackage{amsmath}
\usepackage{etoolbox}
\usepackage{thmtools}
\usepackage[usenames,dvipsnames]{color}
\usepackage[kerning=true]{microtype}
\usepackage[all,cmtip]{xy}
\usepackage{pgf,tikz}
\usetikzlibrary{trees}
\usepackage{wrapfig}
\usepackage{colortbl}
\usepackage{hhline}
\usepackage{eurosym}
\usetikzlibrary{decorations.markings,shadows}
\usetikzlibrary{arrows,shapes,positioning,calc,shadows}
\tikzstyle arrowstyle=[scale=1]
\tikzstyle directed=[postaction={decorate,decoration={markings,
    mark=at position .5 with {\arrow[arrowstyle]{stealth}}}}]
\usetikzlibrary{shapes.geometric}
\usepackage{fullpage}
\usepackage{amsthm}
\usepackage{mathabx}
\usepackage[colorlinks=true, citecolor=cyan,urlcolor=blue,linktocpage=true,
	pagebackref,hyperindex=true,hypertexnames=true,hyperfigures=true]{hyperref} 
\usepackage[font=small,figurename=Fig.]{caption}
\usepackage{ytableau}
\usepackage{genyoungtabtikz}
\usepackage{amssymb}
\YFrench
\Yboxdim{1cm}
\Yfillcolour{green!80}
\definecolor{azure}{rgb}{0.,0.5,1.0} 
\definecolor{green}{rgb}{0.,0.6,0.} 
\newcommand\yred{\Yfillcolour{red}}
\newcommand\ylw{\Yfillcolour{yellow}}

\newcommand\bleu{\textcolor{blue}}
\newcommand\black{\textcolor{black}}
\newcommand\rouge{\textcolor{red}}
\definecolor{jaunefonce}{rgb}{1,.6,0}

\definecolor{vertfonce}{rgb}{0.,0.5,0.} 
\renewcommand\vert{\textcolor{vertfonce}}

\declaretheoremstyle[
  spaceabove=6pt, spacebelow=6pt,
  headindent=0pt,
  headfont=\normalfont\bfseries,
  notefont=\mdseries, notebraces={(}{)},
    bodyfont=\itshape,
  postheadspace=1em]{mystyle}

\declaretheorem[numberwithin=section,style=mystyle]{lemma}

\numberwithin{equation}{section}
\numberwithin{theorem}{section}
\numberwithin{figure}{section}
\numberwithin{table}{section}
\AfterEndEnvironment{proposition}{\noindent\ignorespaces}
\AfterEndEnvironment{lemma}{\noindent\ignorespaces}
\newtheorem{conjecture}{\bleu{Conjecture}}

\makeatletter
\renewenvironment{proof}[1][\proofname]{\par
  \pushQED{\qed}%
  \normalfont \topsep6\p@\@plus6\p@\relax
  \trivlist
   \item[\hskip\labelsep
        \scshape
    \vert{#1}\@addpunct{.}]\ignorespaces
}{%
  \popQED\endtrivlist\@endpefalse
}
\makeatother

\declaretheoremstyle[
spaceabove=3pt, spacebelow=3pt,
  headindent=0pt,
  headfont= \normalfont\bfseries,
  notefont=\mdseries, notebraces={(}{)},
    bodyfont=\normalfont,
  postheadspace=1em,
qed={\tiny\bleu{$\blacksquare$}}
]{probstyle}
\theoremstyle{definition}

\renewcommand*\backref[1]{\ifx#1\relax \else (On page #1) \fi}
 
\newcommand\auteur[1]{{\sc #1}}
\newcommand\define[1]{\bleu{\bf #1}}
\newcommand{\pref}[1]{{\rm (\ref{#1})}}

\newcommand\titreref[1]{{\em #1}}
\newcommand\vol[1]{{\bf #1}}
\newcommand\wrapcaption[1]{\captionsetup{font=tiny}
\caption{#1.}
\captionsetup{font=small}}

\newcommand\area{\mathrm{area}}
\newcommand{\A}{\mathcal{A}}
\renewcommand{\a}{\boldsymbol{a}}

\newcommand{\Cat}{\mathcal{A}}

\newcommand{\des}{\mathrm{des}}

\newcommand\xkn[2]{\mathcal{X}^{(#1,#2)}}
\newcommand\Dk[1]{\mathcal{X}_{#1}}

\newcommand\End{\mathrm{End}}

\newcommand\hook{\varepsilon}

\newcommand\HSeries{\mathbb{H}}
\newcommand\Hall{\mathcal{E}}
\newcommand{\HallOper}[3]{#1\{\xkn{#2}{#3}\}}

\newcommand\Id{\mathrm{Id}}
\newcommand\Lie[2]{\bleu{\Big[ \black{#1}\,,\, \black{#2} \Big]}}
\newcommand\MacH{\widetilde{H}}

\newcommand\N{{\mathbb N}}

\newcommand{\qbinom}[2]{\genfrac{[}{]}{0pt}{}{#1}{#2}}
\newcommand\Q{\mathbb{Q}}

\newcommand\Rational{\mathbb{Q}}

\newcommand\simd{\mathrm{sim}}
\newcommand\Sim{\mathrm{Sim}}

\renewcommand\S{{\mathbb S}}

\newcommand{\unepart}{\hbox{\tiny\,\text \euro\,}}

\newcommand\x{\boldsymbol{x}}
\newcommand\y{\boldsymbol{y}}

\newcommand\Z{\mathbb{Z}}
\definecolor{GREEN}{rgb}{0.,0.6,0.} 

\title{
	{\Large{\bleu{Macdonald polynomials and operators\\
	    and Catalan Combinatorics}}\\
	{\small(Notes for the Link Homology AIM Research Community)}}\\
}
\date{\today}

\author{ \textcolor{green}{Fran\c{c}ois Bergeron}}

\begin{document}

\maketitle

{ \setcounter{tocdepth}{1}\parskip=0pt\footnotesize \tableofcontents}
\parskip=8pt  

%%%%%%%%%%%%%%%%%%%%%%%%%%%%%%%
\section{Introduction} \label{sec_intro}
Our main aim with these notes is to introduce the combinatorial and symmetric function tools that relate to the description of the Poincaré polynomial of the triply graded Khovanov-Rozansky homology of torus links, {\sl a.k.a.} the (reduced) ``superpolynomial'' of these links. There are different normalizations for these polynomial, each natural in the setup they arise in. The one that we will use is chosen to make interesting ``combinatorial'' features stand out\footnote{For instance related to conjectures in~\cite{DGR}.}. In particular, our version of the polynomials lie in $\mathbb{N}[q,t,a]$, and are symmetric in the $q$ and $t$ variables. To help the reader coming from different backgrounds, hence used the to specific convention of his own subject, let us immediately mention a few explicit instances of these polynomials (using our normalization) for some small classical $(k,n)$-torus knots: here denoted $\mathcal{P}_{kn}=\mathcal{P}_{kn}(q,t;\a)$. We have:
\begin{enumerate}\itemsep=4pt
\item $\mathcal{P}_{32}=(q+t)+\a$, \qquad (the trefoil);
\item $\!\begin{aligned}[t]\mathcal{P}_{43}&= (q^3+q^2t+qt^2+t^3+qt)+(q^2+qt+t^2+q+t)\,\a+\a^2\\
    &=(s_3+s_{11})+(s_1+s_2)\,\a +\a^2;\end{aligned}$
\item $\mathcal{P}_{54}=(s_{31} + s_{41} + s_{6}) + (s_{11} + s_{21} + s_{31} + s_{3}  + s_{4} + s_{5})\,\a + (s_{1} + s_{2} + s_{3})\,\a^{2} + \a^{3}$;
\item $\!\begin{aligned}[t]\mathcal{P}_{56}=&\ (s_{43} + s_{42} + s_{62} +  s_{61}  + s_{71} + s_{81} + s_{(10)}) \\
&+ (s_{33}+ s_{32} + s_{42} + s_{52}  +  s_{31} + 2s_{41} + 2s_{51} + 2s_{61} + s_{71} + s_{6} + s_{7} + s_{8} + s_{9})\,\a\\ 
&+ (s_{32} +s_{11} + s_{21}  + 2s_{31}  + s_{41} + s_{51}  + s_{3}+  s_{4}+ 2s_{5} + s_{6} + s_{7})\,\a^{2} \\ 
&+ (s_{1} + s_{2} + s_{3} + s_{4})\,\a^{3} + \a^{4}.\end{aligned}$ 
\end{enumerate} 
Here, in order to make the description of mangeable size\footnote{The $(q,t,\a)$-expansion of $\mathcal{P}_{56}$ has $197$ terms.}, we have ``compressed'' the description using the Schur polynomials:
	$$s_{ab}=s_{ab}(q,t):= q^at^b+q^{a-1}t^{b+1}+\ldots + q^{b+1} t^{a-1}+q^bt^a =\frac{q^{a+1}t^b-q^b t^{a+1}}{q-t},$$
assuming that $a\geq b\geq 0$. This presentation makes apparent one of the (conjectural) properties of the polynomials $\mathcal{P}_{nk}$, namely that they appear to be Schur positive in $q,t$.

We will describe at least two different approaches to the above polynomials: one via operators on symmetric functions, and another that exploits specific combinatorial constructions. A first instance of the operators approach makes use of ``Macdonald eigenoperators'', which is to say that their eigenfunctions are the (modified) Macdonald polynomials described in the next section. This makes it possible to describe all the superpolynomials for any $(n,kn+1)$-torus knots. 

To generalizes these considerations to all torus knots and links, we  next consider an operator realization of the elliptic Hall algebra (see \autoref{sec_elliptic}), which is essentially described as the enveloping algebra of the Lie-bracketing closure of four simple basic operators on symmetric functions. As we will explain in \autoref{sec_homology}, elements of the resulting algebra are then used as ``creation'' operators to build up symmetric functions $\mathcal{F}_{nk}$, with coefficients in the fraction field $\mathbb{Q}(q,t)$. The $(q,t)$-coefficients of power of $a$ in the superpolynomials $\mathcal{P}_{nk}$ then appear as coefficients of hook-indexed functions in these $\mathcal{F}_{nk}$.

A parallel story, goes through a purely combinatorial description of the $\mathcal{P}_{nk}$, via constructions on (integer) partitions contained in a ``triangular'' partition\footnote{See \autoref{sec_combinatorial} for definitions.} associated to the pair of integers $n$ and $k$. The number of such partitions, for the case $k=n+1$, is the well known Catalan number $\frac{1}{n+1}\binom{2n}{n}$, and many of the statistic and construction involved in the story corresponds to classical notions of ``Catalan'' combinatorics. Nowadays, this subject includes many variations on the theme, including the study of constructions on partitions included inside any fixed triangular one. On the relevant partitions, we will describe the two statistics: ``area'' and ``dinv''. And each $(q,t)$ coefficients of the polynomials $\mathcal{P}_{nk}(q,t;a)$ will be described as the weighted enumeration of the partitions, with weight $q^{\mathrm{area}}$ and $t^{\mathrm{dinv}}$. We will aim to make the description of these statics as natural as possible, particularly the ``dinv'' statistic which often appears to have a strange definition. We believe that our approach will remove much of this strangeness. 

It would be interesting to add to this the combinatorial connections to the Hilbert Scheme side of the story. Such connections are very nicely described by Eugene Gorsky in his series of lectures.

%%%%%%%%%%%%%%%%%%%%%%%%%%%%%%%
\section{Plethysm}\label{SecPlethysm}
Let $\Lambda_{q,t}$ denote the algebra of symmetric functions\footnote{Using this terminology to stress that we are working with ``polynomials'' in infinitely many variables.} in the variables $\x=x_1,x_2,x_3,\ldots$ over the fraction field $\mathbb{Q}(q,t)$.
Our calculations will be greatly simplified using \define{plethystic operations}. This corresponds\footnote{Technically, it would be best to say that we have a $\lambda$-ring structure on $\mathcal{A}$, see \href{https://en.wikipedia.org/wiki/Lambda-ring}{lambda-ring}.} to a specific action of $\Lambda_{q,t}$ on a algebra $\mathcal{A}$ of series or polynomials (over $\mathbb{Q}(q,t)$). In simple terms, if we denote by $f[\boldsymbol{a}]$ the effect of $f\in\Lambda_{q,t}$  on $\boldsymbol{a}\in \mathcal{A}$, we have:
\begin{align}\itemsep=6pt
&\bleu{0[\boldsymbol{a}]}=\bleu{0},\qquad {\mathrm{and}}\qquad \bleu{1[\boldsymbol{a}]}=\bleu{\boldsymbol{a}},\\
&\bleu{\big(\sum_i \alpha_i f_i\big)[\boldsymbol{a}]}=\bleu{\sum_i  \alpha_i f_i[\boldsymbol{a}]},\\
&\bleu{\big(\prod_i \alpha_i f_i\big)[\boldsymbol{a}]}=\bleu{\prod_i  \alpha_i f_i[\boldsymbol{a}]}.
\end{align}
with the $\alpha_i$ denoting scalars. Since $\Lambda_{q,t}$ is generated (as an algebra) by the power-sum symmetric functions
$p_k(\x)=\sum_i x_i^k$, it follows that all calculations may be reduced to $p_k[\boldsymbol{a}]$ via the above rules. The operators\footnote{In the context of $\lambda$-rings, these are the Adams operators.} $p_k$ are especially nice, since they have the following properties which may be used to turn the definition into a calculation algorithm.
\begin{align}\itemsep=6pt
{\rm{(1)}}\quad &\textstyle \bleu{p_k[\sum_i \boldsymbol{a}_i]}= \bleu{\sum_i p_k[\boldsymbol{a}_i]},           
	&{\rm{(2)}}&\quad \textstyle \bleu{p_k[\prod_i \boldsymbol{a}_i]}= \bleu{\prod_i p_k[\boldsymbol{a}_i]},\nonumber\\
{\rm{(3)}}\quad &\bleu{p_k[ \boldsymbol{a}/\boldsymbol{b}]}=\bleu{p_k[\boldsymbol{a}]/p_k[\boldsymbol{b}]},          
	&{\rm{(4)}}&\quad \bleu{p_k[p_j]}=\bleu{p_{kj}},\label{def_plethysme}\\
{\rm{(5)}}\quad &\bleu{p_k[\boldsymbol{a}\otimes \boldsymbol{b}]}=\bleu{p_{k}[\boldsymbol{a}]\otimes p_{k}[\boldsymbol{b}]},              
	&{\rm{(6)}}&\quad \bleu{p_k[\varepsilon]} =\bleu{(-1)^k},\nonumber\\
{\rm{(7)}}\quad &\bleu{p_k[x]=x^k},\ \hbox{if}\  \bleu{x}\  \hbox{is a variable},                   
&{\rm{(8)}}&\quad \bleu{p_k[c]=c},\ \hbox{if}\  \bleu{c}\  \hbox{is  a constant}.\nonumber
\end{align}
In simple situations, when $\boldsymbol{a}$ is a positive sum of monomials, the plethysm $f[\boldsymbol{a}]$ corresponds to the ``evaluation'' of $f$ in the ``monomials'' of $\boldsymbol{a}$. Let us illustrate all this with some examples.

Starting with a very simple case, it is handy to consider sets of variables as formal sums, setting $\x=\sum_i x_i$ or $\y=\sum_i y_i$. Observe that this makes it so that $f[\x]=f(\x)$, since
     $$\bleu{p_k[x_1+x_2+\ldots]}=\bleu{x_1^k+x_2^k+\ldots} = \bleu{p_k(\x)}.$$
 The ``classical'' comultiplication on symmetric function then corresponds (up to a direct traduction) to the calculation of $f[\x+\y]$. More precisely, for the comultiplication $\Delta:\Lambda \rightarrow \Lambda\otimes \Lambda$  on $f$ expands in the Schur basis as
    \begin{equation}\label{f_coproduct}
       \bleu{\Delta f} =\bleu{ \sum_{\lambda,\mu} c_{\lambda,\mu}\, s_\lambda\otimes s_\mu},
    \end{equation}
if and only if
     \begin{equation}
       \bleu{f[\x+\y] } =\bleu{ \sum_{\lambda,\mu} c_{\lambda,\mu}\, s_\lambda(\x) s_\mu(\y)}.
    \end{equation}
  The \define{primitive elements} for $\Delta$ are the $p_k$, since $p_k[\x+\y] = p_k[\x]+p_k[\y] $. Plethysms may also involve series of symmetric functions. For example, consider
     \begin{align}
           \bleu{\HSeries(\x)}&:=
         \bleu{\sum_{n\geq 0} h_n(\x)}
          =\bleu{\exp\!\Big(\sum_{k\geq 1}\frac{p_k(\x)}{k}\Big)}.\label{H_series}\\
       \bleu{\mathbb{E}(\x)}&:=
         \bleu{\sum_{n\geq 0} e_n(\x)}
          =\bleu{\exp\!\Big(\sum_{k\geq 1}(-1)^{k-1} \frac{p_k(\x)}{k}\Big)}.\label{E_series}
       \end{align}
 Then we have
     \begin{equation}
     	\bleu{\HSeries[\x+\y]} 
       		=\bleu{\HSeries(\x) \, \HSeries(\y)},
	\qquad {\rm and}\qquad \bleu{\mathbb{E}[\x+\y]} 
       		=\bleu{\mathbb{E}(\x) \, \mathbb{E}(\y)},
    \end{equation}
Comparing terms of same degree in the identity $\HSeries[\x]\,\mathbb{E}[-\x]=1$, we get
  \begin{align}
  	 &\bleu{e_n[-\x]}=  \bleu{(-1)^n h_n(\x)},\\
          &\bleu{e_n[\x+\y] }=  \bleu{\sum_{k=0}^n e_k(\x)e_{n-k}(\y)}.
  \end{align}
Let us next consider the calculation of $s_2[\x\y]$ with the aim of writing the answer in terms of Schur polynomials. Since 
	$$s_2=\frac{1}{2}(p_1^2+p_2),\qquad{\rm and}\qquad s_{11}=\frac{1}{2}(p_1^2-p_2),$$
we calculate that
\begin{align*}
s_2[\x\cdot \y]
	&=\frac{1}{2}(p_1^2+p_2)[\x\cdot\y]\\
	&=\frac{1}{2}(p_1^2[\x\cdot\y]+p_2[\x\cdot\y])\\          
	&=\frac{1}{2}(p_1(\x)^2\, p_1(\y)^2+p_2(\x)\,p_2(\y))\\ 
	&=s_2(\x) s_2(\y) +s_{11}(\x) s_{11}(\y),                                  
\end{align*}
In general, using the identity, we get
\begin{align} 
	\bleu{h_n[\x\cdot\y]} 
	 	&= \bleu{\sum_{\lambda \vdash n} \frac{1}{z_\lambda} p_\lambda(\x) p_\lambda(\y)}\\
		&= \bleu{\sum_{\lambda \vdash n} s_\lambda(\x) s_\lambda(\y)}.\label{cauchy_ss}
	\end{align}
See \autoref{sec_sym} for more identities related to this last calculation. In particular, we have the identities:
   \begin{equation}
   	\bleu{\HSeries(a\x)^\perp f(\x)} = \bleu{f[\x+a]},\qquad {\rm and}\qquad \bleu{\mathbb{E}(a\y)^\perp f(\x)} = \bleu{f[\x-\varepsilon a]},
  \end{equation}
  where $G^\perp$ stands for the adjoint of multiplication by $G$ for the Hall scalar product, and $f$ is any symmetric function. Unfolding the second identity, we have
\begin{equation}\label{e_skewing}
   \bleu{ \sum_{k\geq 0} a^k (e_k^\perp f)(\x)} 
   	= \bleu{f[\x-\varepsilon a]}.
 \end{equation}
Recall that Schur functions are orthonormal for the Hall scalar product. The explicit effect of $e_k^\perp$ on a given Schur function $s_\lambda=s_\lambda(\x)$ is given by the \define {dual Pieri rule}
\begin{equation}\label{e_Pieri}
   \bleu{e_k^\perp s_\lambda}= \bleu{\sum_{\lambda/\mu\in\mathrm{VS}(k)} s_\mu},
 \end{equation}	
 where the sum is over the set of partitions $\mu$ obtained by removing a ``vertical strip'' of size $k$ from $\mu$. This is a set of $k$ boundary cells of $\lambda$, no two of which lying on the same row. For example, as illustrated in \autoref{figure_strip}, the partition $\mu=32$ is obtained from $4211$ by removing a vertical strip of size $3$.
 \begin{figure}[ht]
 \begin{tikzpicture}
 \Yboxdim{10pt}
  \Yfillcolour{yellow}
 \tgyoung(0cm,0cm,;;;!\yred;!\ylw,;;!\yred,;,;)
  \node at (2.5,.4) {\Large$\longmapsto$};
    \Yfillcolour{yellow}
  \tyng(3.5cm,0cm,3,2)
 \end{tikzpicture}
\caption{Removing a vertical strip of size $3$.}
\label{figure_strip}
\end{figure}
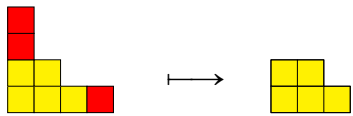
 For more on plethysm, see~\cite{livre,haglund}. 

As a last illustration of how plethysm works, let us expand $h_n(\x)$ in terms of the basis\footnote{As we will see, this is an interesting specialization of Macdonald polynomials, up to a multiplicative factor.} $s_\mu[\x/(1-q)]$. 
Considering \autoref{cauchy_ss}, we  get 
 $$h_n[\x] =  h_n\big[ (1-q)\,{\textstyle\frac{\x}{1-q}}\big]=\sum_{\mu\vdash n} s_\mu[1-q]\, s_\mu\big[{\textstyle\frac{\x}{1-q}}\big].$$
By direct calculation in \autoref{defn_Omega}, we get $\HSeries[(1-q)z]=(1-qz)/(1-z)$. Hence $h_0[1-q]=1$ and $h_k[1-q]= 1-q$ when $k>0$.  It follows that, at $(1-q)$ the Jacobi-Trudi determinant (see \autoref{formule_jacobi}) most often evaluates to $0$ (since we have two rows equal), except when $\mu=(a\,|\,\ell)$ is a hook (one part of size $a+1$ , and $\ell$ parts of size $1$). 
When this is so, one recursively shows that the determinant is equal to $(-q)^\ell(1-q)$.  We thus conclude that
  \begin{align}
     \bleu{h_n[\x]} &= \bleu{\sum_{\ell=0}^{n-1} s_{(a\,|\,\ell)}[1-q] \, s_{(a\,|\,\ell)}\textstyle\big[{\x}/{(1-q)}\big]}\\
     			  &= \bleu{\sum_{\ell=0}^{n-1} (-q)^\ell(1-q) \, s_{(a\,|\,\ell)}\textstyle\big[{\x}/{(1-q)}\big]}.
  \end{align}
Equivalently,
  \begin{equation}\label{hook_eval_1q}
     \bleu{s_\mu[1-q]} =\begin{cases}
      \bleu{ (-q)^\ell(1-q)}& \text{if}\ \mu=(a\,|\,\ell), \\
      \bleu{ 0} & \text{otherwise}.
\end{cases}
  \end{equation}
  It follows that 
    \begin{equation}\label{eval_hook_1a}
     \bleu{\frac{s_{(a\,|\,\ell)}[1-\varepsilon\,q]}{1+q}} = \bleu{q^\ell}.
     \end{equation}

%%%%%%%%%%%%%%%%%%%%%%%%%%%%%%%
\section{Macdonald polynomials} \label{sec_macdonald}
As alluded to in the introduction, our approach to the calculation of the superpolynomials for link homology involves operators defined in terms of ``modified'' Macdonald polynomials, denoted\footnote{The use of tilde is inherited from a long tradition.} by $\MacH_\mu=\MacH_\mu(q,t;\x)$, a renormalized version of the \define{original Macdonald polynomials} that are usually denoted by $P_\mu$. Rather than proceed historically, let us skip directly to definition of the version that we will use. The set $\{\MacH_\mu(q,t;\x)\}_\mu$ is linear basis of the algebra $\Lambda_{q,t}$ of symmetric functions\footnote{Using this terminology to stress that we are working with ``polynomials'' in infinitely many variables.} in the variables $\x=x_1,x_2,x_3,\ldots$ over the fraction field $\mathbb{Q}(q,t)$. 

Recall that ${\lambda \prec \mu}$ stands for 
$\lambda$ being (strictly) smaller than $\mu$ in the \define{dominance order}, which is defined by 
          $$\bleu{\lambda\preceq \mu}\qquad  \hbox{ iff \qquad}\qquad \forall i\qquad \bleu{\sum_{j=1}^{{i}}  \lambda_j \leq \sum_{j=1}^{{i}}  \mu_j },$$
adding $0$ parts if necessary to make sense of the above. 
The \define{(modified\footnote{We sometimes use this terminology to underline that they differ from the classical one, essentially via a simple plethysm.}) Macdonald polynomial}, $\MacH_\mu=\MacH_\mu(q,t;\x)$ are homogeneous symmetric functions of degree $n=|\mu|$, uniquely characterized by the equations
\begin{enumerate}\itemsep=6pt
   \item[(1)]  $\bleu{\langle \MacH_\mu(q,t;\x),s_n(\x)\rangle=1}$,
  \item[(2)]  $\bleu{\langle {\MacH_\mu[q,t;\x\,(1-q)]}),s_\lambda(\x)\rangle=0}$, for all $\lambda\vdash n$ such that $\bleu{ \lambda\not\succeq \mu}$ in dominance order,
   \item[(3)]  $\bleu{\langle{\MacH_\mu[q,t;\x\,(1-t)]},s_\lambda(\x)\rangle=0}$, for all $\lambda\vdash n$ such that  $\bleu{ \lambda\not\succeq \rouge{\mu'}}$.
\end{enumerate} 
For $n=2$, the first condition imposes that
   $$\MacH_2 = s_2+a s_{11},\qquad {\rm and} \qquad \MacH_{11} = s_2+b s_{11},$$
  for some $a,b$ in $\Q(q,t)$. Conditions (2) and (3) respectively state that
     $$\langle \MacH_2[\x\,(1-\bleu{q})],s_{11}\rangle = 0,\qquad {\rm and} \qquad \langle \MacH_{11}[\x\,(1-\bleu{t})],s_{2}\rangle = 0,$$
  which implies that $a=q$ and $b=t$, hence   
$$
\MacH_{{2}}(\bleu{q},\bleu{t};\x)=s_{{2}}(\x)+\bleu{q}\,s_{{11}}(\x),\qquad {\rm and}\qquad 
\MacH_{{11}}(\bleu{q},\bleu{t};\x)=s_{{2}}(\x)+\bleu{t}\,s_{{11}}(\x).
$$
We may also calculate that 
\begin{eqnarray*}
&&\MacH_{{3}}(\bleu{q},\bleu{t};\x)=s_{{3}}(\x)+ \left(\bleu{ {q}^{2}+q }\right) s_{{21}}(\x)+\bleu{{q}^{3}}s_{{111}}(\x),\\
&&\MacH_{{21}}(\bleu{q},\bleu{t};\x)=s_{{3}}(\x)+ \left(\bleu{ q+t }\right) s_{{21}}(\x)+\bleu{qt}\,s_{{111}}(\x),\\
&&\MacH_{{111}}(\bleu{q},\bleu{t};\x)=s_{{3}}(\x)+ \left(\bleu{ {t}^{2}+t} \right) s_{{21}}(\x)+\bleu{{t}^{3}}s_{{111}}(\x);
\end{eqnarray*}
 The coefficients of the expansion $\MacH_\lambda(q,t;\x)=\sum_\mu \widetilde{K}_{\lambda\mu}(q,t)\,s_\mu(\x)$, are the \define{$(q,t)$-Kostka polynomials}. We $n=4$ we get the matrix
 $$(\widetilde{K}_{ \lambda\mu}(q,t))_{\lambda,\mu\vdash 4}=\left(\begin{array}{rrrrr}
1 & q^{3} + q^{2} + q & q^{4} + q^{2} & q^{5} + q^{4} + q^{3} & q^{6} \\
t & q^{2} t + q t + 1 & q^{2} t + q & q^{3} t + q^{2} + q & q^{3} \\
t^{2} & q t^{2} + q t + t & q^{2} t^{2} + 1 & q^{2} t + q t + q & q^{2} \\
t^{3} & q t^{3} + t^{2} + t & q t^{2} + t & q t^{2} + q t + 1 & q \\
t^{6} & t^{5} + t^{4} + t^{3} & t^{4} + t^{2} & t^{3} + t^{2} + t & 1
\end{array}\right),$$
with row and column indexed by partitions in ``lexicographic order'' ({\sl i.e.}: $4,31,22,211,1111$).
Setting both $q$ and $t$ equal to $1$, we see that $\MacH_\mu(1,1;\x)=h_1(\x)^n$ for all $\mu\vdash n$.
 One may observe many (well-known) symmetries in this matrix. A first symmetry, easily deduces from the above characterization, corresponds to 
 \begin{equation}
     \bleu{\MacH_{\rouge{\mu'}}(q,t;\x):=\MacH_\mu(\rouge{t},\rouge{q};\x)}.
 \end{equation}
In order to describe a second symmetry, let us consider the involutive ``$\downarrow$'' operator\footnote{This operator is denoted $\bar\omega$ in the proof of the shuffle conjecture by Carlsson-Mellit.}, on symmetric functions $f(q,t;\x)$ over the field $\Rational(q,t)$, defined as $\downarrow f( q,t;\x):= \omega\, f(\x;q^{-1},t^{-1})$.   Then, it follows directly from results of~\cite{macdonald_lotha} that
  \begin{equation}\label{mac_sym2}
     \bleu{q^{\eta(\mu')}t^{\eta(\mu)} \downarrow \MacH_\mu(q,t;\x):=\MacH_\mu(q,t;\x)},
 \end{equation}
 where $\eta(\mu')$ and $\eta(\mu)$ are jointly defined by the vector identity
   $$\bleu{(\eta(\mu'),\eta(\mu')}=\bleu{\sum _{(i,j)\in\mu} (i,j)}.$$
Here $(i,j)\in \mu$ stands for $(i,j)$ being a cell\footnote{Cells are encoded by the cartesian coordinates of their south-west vertex.} of $\mu$. For instance, the cells of the partition  $\mu=32$ are
\begin{center}
 \Yboxdim{22pt}
 \Ylinecolour{blue}
\Ynodecolour{green!50!black}
\Yfillcolour{yellow}
\begin{tiny}
 \young(<{(0,0)}><{(1,0)}><{(2,0)}>,<{(0,1)}><{(1,1)}>)
 \end{tiny}
\end{center}
The specialization at $t=1$ of Macdonald polynomials
expand ``multiplicatively'' as
   \begin{align} \label{special1}
      &\bleu{\MacH_\mu(q,1;\x)}
           =\bleu{\frac{h_\mu\big[{\x}/{(1-q)}\big]}{h_\mu\big[{1/{(1-q)}\big]}}}.
   \end{align}
On the other hand, at $t=1/q$, we get    
   \begin{align}\label{special1_q}
      &\bleu{\MacH_\mu(q,1/q;\x)}=\bleu{\frac{s_\mu\big[{\x}/{(1-q)}\big]}{s_\mu\big[{1/{(1-q)}\big]}}}.
   \end{align}
The \define{$\star$-scalar product} is defined on the power-sum basis by
    \begin{equation}
      \bleu{\langle p_\mu,p_\lambda\rangle_\star }:= \begin{cases}
     \bleu{Z_\mu(q,t)}  & \text{if}\ \bleu{\mu=\lambda}, \\
     \bleu{0} & \text{otherwise},
\end{cases} 
    \end{equation}
  where $ Z_\mu(q,t) :=(-1)^{|\mu|-l(\mu)}\,p_\mu[(1-q)\,(1-t)]\,z_\mu$.
Observe that it is related to the Hall scalar product by
\begin{equation}
  \bleu{\langle f,g\rangle = \langle f,\omega\, \widehat{g}\rangle_\star }, 
 \end{equation}
where $\widehat{g}(\x):= \textstyle g\!\left[\frac{\x}{(1-t)(1-q)}\right]$.
In particular, $\widehat{p_k}(\x)= p_k(\x)/((1-q^k)(1-t^k))$.

%%%%%%%%%%%%%%%%% 
\subsection*{Cauchy Kernel and scalar product}
We set
	$$ \bleu{w_\mu(q,t)}:=\bleu{\prod_{c\in \mu} (q^{a(c)} -t^{l(c)+1}) (t^{l(c)} -q^{a(c)+1})},\qquad {\rm and}\qquad 
	 \bleu{\widehat{H}_\mu}:=\bleu{\frac{1}{w_\mu(q,t)} \MacH_\mu},$$
with $a(c)=a_\mu(c)$ (resp. $\ell(c)=\ell_\mu(c)$) standing for the \define{arm} (resp. \define{leg}) of the cell $c$ in $\mu$.
The Cauchy-kernel formula for the modified Macdonald states that
\begin{equation}\label{qt_kernel}
   \bleu{e_n\!\textstyle\left[\frac{\x\cdot\y}{(1-q)(1-t)}\right]} = \bleu{\sum_{\mu\vdash n} \MacH_\mu(\x)\widehat{H}_\mu(\y)}.
 \end{equation}
The Cauchy-kernel formula is equivalent to
 $$\bleu{\langle \MacH_\lambda,\widehat{H}_\mu\rangle_\star}=\bleu{\delta_{\lambda,\mu}}.$$
 In other words, the Macdonald polynomials form an orthogonal family for $\star$-scalar product. It follows from \autoref{qt_kernel}, and some calculations, that 
 \begin{align}
   \bleu{e_n(\x)} &= \bleu{\sum_{\mu\vdash n} \widehat{H}_\mu[(1-t)(1-q)]\,\MacH_\mu(\x)}\label{Mac_expansion_of_en}\\
      &=\bleu{\sum_{\mu\vdash n} (1-t)(1-q)\frac{B_\mu(q,t)\Pi_\mu(q,t) }{w_\mu(q,t)}\,   \MacH_\mu(\x)}\label{Mac_expansion_of_enqt},
 \end{align}
 where $B_\mu(q,t)$ is the \define{cell enumerator} of $\mu$:
 \begin{equation}
    \bleu{B_\mu(q,t)}:=\bleu{\sum_{(i,j)\in \mu} q^it^j}.
 \end{equation}
To finish parsing \autoref{Mac_expansion_of_enqt}, we also define
  \begin{equation}\label{def_Pi_mu}
    \bleu{\Pi_\mu(q,t)}:=\bleu{\prod_{(i,j)\neq (0,0)} (1-q^it^j)}.
 \end{equation}
For the partition  $(3,2)$, then we have $B_{32}(q,t) = 1+q+q^2+t+qt$, which is the sum of the monomials arising in the filling
\begin{center}
 \Yboxdim{15pt}
 \Ylinecolour{blue}
\Ynodecolour{green!50!black}
\Yfillcolour{yellow}
 \young(<1><q><q^2>,<t><qt>)
\end{center}
Moreover, $\Pi_{32}(q,t)=(1-q)(1-q^2)(1-t)(1-qt)$.

%%%%%%%%%%%%%%%%%%%%%%%%%%%%%%% 
 \section{Macdonald Eigenoperators}\label{SecMacOper}
 By definition,  \define{Macdonald eigenoperators} afford as common eigenfunctions the Macdonald polynomials. Among the most interesting ones are those for which the eigenvalue of $\MacH_\mu$ are symmetric functions of the values $q^it^j$, for $(i,j)$ varying in the set of cells of the partition $\mu$. More precisely, for a given symmetric function $f$, we set 
 \begin{equation}\label{def_delta_oper}
     \bleu{f[B]\, \MacH_\mu}:=\bleu{f\big[B_\mu(q,t)\big]\,\MacH_\mu},\qquad {\rm and}\qquad 
      \bleu{f[\overline{B}]\, \MacH_\mu}:=\bleu{f\big[B_\mu(q^{-1},t^{-1})\big]\,\MacH_\mu},
 \end{equation}
In particular $B$ is the Macdonald eigenoperator that sends $\MacH_\mu$ to $B_\mu(q,t)\MacH_\mu$. We also consider the simple operators:
 \begin{equation}\label{def_B_oper}
     \bleu{M\,\MacH_\mu}:= \bleu{(1-q)(1-t)\MacH_\mu},\qquad {\rm and}\qquad 
      \bleu{\overline{M}\MacH_\mu}:= \bleu{(1-q^{-1})(1-t^{-1})\MacH_\mu}.
   \end{equation}
The $f[B]$ (resp. $f[B-1]$) operator has often been denoted by $\Delta_{f}$  (resp. $\Delta'_{f}$). However, it seems more natural to adopt a plethystic style notation for Macdonald eigenoperators, mimicking the behavior of their eigenvalues. We have an obvious homomorphism from the ring $\Lambda_{\x}$ to the ring of  Macdonald eigenoperators:
 \begin{eqnarray*}
     &&\bleu{1[B]}=\bleu{\Id},\\
     &&\bleu{(f\pm g)[B]}=\bleu{f[B]\pm g[B]},\\
     &&\bleu{(f\cdot g)[B]}=\bleu{f[B]\cdot g[B]}.
 \end{eqnarray*}
Two other interesting similar homomorphisms are 
  \begin{equation}
     \bleu{f\mapsto (\omega f)[B-1/M]},\qquad {\rm and}\qquad \bleu{f\mapsto (\omega f)[\overline{B}-1/\overline{M}]}.
  \end{equation}
A ``special case'' is the \define{nabla operator} $\nabla$:
 \begin{equation}
     \bleu{\nabla(\MacH_\mu):=T_\mu(q,t) \,\MacH_\mu}, \qquad {\rm with}\qquad \bleu{T_\mu(q,t):=\prod_{i,j} q^it^j= q^{\eta(\mu')} t^{\eta(\mu)}}.
  \end{equation}
It is easy to observe, using \autoref{mac_sym2}, that $\bleu{\nabla^{-1}=\ \downarrow \nabla \downarrow}$. Moreover, when restricted to homogeneous symmetric functions of degree $n$, we have the operator equality $\nabla=e_n[B]$.

Sometimes the effect of a Macdonald operator $f[B]$ on a specific symmetric function $g(\x)$ is also symmetric in $q$ and $t$. When this is the case, the coefficients of each Schur function in $f[B](g(\x))$ are often themselves Schur positive polynomials in the ``parameters'' $q,t$. In particular, this is so for all the symmetric functions $\nabla^m(e_n)$, which arises as the bigraded Frobenius characteristic (defined below) of the $\S_n$-module of higher diagonal coinvariants. 
For instance, we have 
\begin{eqnarray*}
&&\nabla(e_{{1}})=1\otimes s_{{1}},\\
&&\nabla(e_{{2}})=1\otimes s_{{2}}+s_1\otimes s_{{11}},\\
&&\nabla(e_{{3}})=1\otimes s_{{3}}+ (s_1+s_e)\otimes s_{{21}}+ 
  (s_{11}+s_3)\otimes s_{{111}},\\
&&\nabla(e_{{4}})=1\otimes s_{{4}}+ (s_1+s_2+s_3)\otimes s_{{31}} + (s_2+s_{21}+s_4)\otimes s_{{22}}\\
&&\qquad\qquad + (s_{11}+s_{21}+s_{31}+s_3+s_4+s_5)\otimes s_{{211
}} + (s_{31}+s_{41}+s_6)\otimes s_{{1111}},
\end{eqnarray*}
where we write $s_\mu\otimes s_\lambda$ for $s_\mu(q,t)\,s_\lambda(\x)$. Observe that, for $a\geq b \geq 0$, we have
\begin{align}
  s_{ab}(q,t) = q^at^b+q^{a-1}t^{b+1}+\ldots +q^{b-1}t^{a+1}+ q^bt^a.
\end{align}
It follows from \autoref{special1}, that at $t=1$,   $\nabla$  becomes multiplicative, hence it is entirely characterized by the fact that 
  \begin{equation}
      \bleu{\nabla_{t=1}\ h_n\!\left[{\x}/{(1-q)}\right]= q^{\binom{n}{2}}\,h_n\!\left[{\x}/{(1-q)}\right],}
   \end{equation}
  since $\MacH_\mu(q,1;\x)$ is proportional to $h_\mu\!\left[{\x}/{(1-q)}\right]$. Since $\MacH_\mu(q,1;\x/q)$ is proportional to $s_\mu\!\left[{\x}/{(1-q)}\right]$, we also have
   \begin{equation}
      \bleu{\nabla_{t=1/q}\ s_\mu\!\left[{\x}/{(1-q)}\right]= q^{\eta(\mu')-\eta(\mu)}\,s_\mu\!\left[{\x}/{(1-q)}\right].}
   \end{equation}

  %%%%%%%%%%%%%%%%%%%%%%%%%%%%%%%
  %%%%%%%%%%%%%%%%%%%%%%%%%%%%%%%
  \section{Elliptic Hall Algebra Operators}\label{sec_elliptic} 
We can now describe an operator realization $\Hall$ of the elliptic Hall algebra over the field $\mathbb{K}=\Rational(q,t)$. This is a $\Z^2$-graded algebra
    \begin{equation}
       \bleu{\Hall}=\bleu{\bigoplus_{(k,n)\in \Z^2} \Hall^{(k,n)}}.
    \end{equation}
Elements of $\Hall^{(k,n)}$ act as operators on the degree-graded algebra $\Lambda=\oplus_{j} \Lambda_j$ of symmetric functions over $\mathbb{K}$, sending the graded component $\Lambda_j$ to $\Lambda_{j+n}$. In view of the following discussion, it is handy to set $\HallOper{\Lambda_d}{a}{b}:=\Hall^{(k,n)}$, with $d=\gcd(k,n)$ and $(k,n)=(ad,bd)$.  
The algebra $\Hall$ is  generated  by certain commutative\footnote{This is one of the basic properties of elliptic algebras, which is here considered as a fact.} subalgebras, 
     \begin{equation}
       \bleu{\HallOper{\Lambda}{a}{b}}=\bleu{\bigoplus_{d\in\N} \HallOper{\Lambda_d}{a}{b}},
    \end{equation}
indexed by coprime integers (corresponding to ``rays'' in the discrete plane), and considered to be graded by $d$.
Each of these graded algebra is isomorphic to the graded algebra $\Lambda$ of symmetric functions, which itself may be identified with $\HallOper{\Lambda}{0}{1}=\Lambda^{\bullet}$ (modulo that we turn any symmetric functions $f$ into the operator $f^\bullet$). We denote by $\HallOper{f}{a}{b}$ the operator associated via this isomorphism to $f\in\Lambda_d$. 

To summarize, $\HallOper{f}{0}{1}=f^\bullet$,  for all $g\in \Lambda_j$ the symmetric function $\HallOper{f}{a}{b}\cdot g$ lies in $\Lambda_{bd+n}$, and we have
\begin{equation}
 \begin{split}
    &\bleu{\HallOper{(\alpha f+\beta g)}{a}{b}} =\bleu{\alpha\, \HallOper{f}{a}{b}+\beta \,\HallOper{g}{a}{b}},\qquad \alpha,\beta\in\mathbb{K}\\
    &\bleu{\HallOper{(f\cdot g)}{a}{b}} =\bleu{\HallOper{f}{a}{b} \circ \HallOper{g}{a}{b}},
 \end{split}
 \end{equation}
 where ``$\circ$'' denotes composition of (commuting) operators.
Hence, for any basis $\big(g_\mu(\x)\big)_\mu$  of $\Lambda_d$, we have
  \begin{equation}\label{IFF_rays}
      \bleu{\HallOper{f}{a}{b}}=\bleu{\sum_{\mu\vdash d} c_\mu(q,t) \HallOper{g_\mu}{a}{b}}\qquad {\rm iff}\qquad 
      \bleu{f(\x)}=\bleu{\sum_{\mu\vdash d} c_\mu(q,t) g_\mu(\x)\in \Lambda_d},
  \end{equation} 
For sure, operators associated to different  coprime pairs will not commute in general.

As it is the only portion that is needed for our purpose,  we now on restrict our description to the positive part of $\Hall$, {\sl i.e.}  $a$ and $b$ are $\geq 1$.
We underline that it is natural to identify  $\HallOper{\Lambda}{0}{1}$ with $\Lambda$.
The most amenable operators for direct calculation are the $\HallOper{\pi_\mu}{a}{b}$ described further below. They are multiplicative, {\sl i.e.:} we have
\begin{equation}
    \bleu{\HallOper{\pi_\mu}{a}{b}} = \bleu{\prod_{k\unepart \mu} \HallOper{\pi_k}{a}{b}},
\end{equation}
for any partition $\mu$. In view of \autoref{IFF_rays}, this will allow us to obtain all operators by setting:
\begin{equation}\label{DefiningRule}
     \bleu{\HallOper{f}{a}{b}}:=\bleu{\sum_{\mu\vdash d} c_\mu(q,t) \, \HallOper{\pi_\mu}{a}{b}},
\end{equation}
via an explicit calculation of the coefficients $c_\mu(q,t)$ of the expansion of $f(\x)$ in the $\pi_\mu(\x)$ basis (which we may identify to $\HallOper{\pi_\mu}{0}{1}$).

\subsection*{Definition of \texorpdfstring{$\HallOper{\pi_d}{a}{b}$}{pi}}   For any pair $(k,n)$, with $k$ and $n$ integers $\geq 1$, consider $\gcd(k,n)=d$  and $(a,b)$ such that $(k,n)=(ad,bd)$.
For the purpose of the following construction, it is handy to write $\xkn{k}{n}$ for $ \HallOper{\pi_d}{a}{b}$. 
For any coprime $a$ and $b$ in $\mathbb{N}^+$, there is a unique pair of ``vectors'' $(r,s)$ and $(u,v)$ in $\N^2$, such that
\begin{enumerate}
 \item $\bleu{(a,b)=(r,s)+(u,v)}$, and
 \item $\bleu{\det \begin{pmatrix} r& s\\ u & v\end{pmatrix} =1}$.
\end{enumerate}
As illustrated in \autoref{Fig_Split}, this corresponds to the pair of integral points that are closest to the line joining the origin to $(a,b)$.  
With this pair at hand, and setting $(r',s'):=(k-r,n-s)$ ,we say that the matrix $\left(\begin{smallmatrix}r&s\\r' &s'\end{smallmatrix}\right)$ \define{split} $(k,n)$. Consider the operator $D_0$, which maps  $\MacH_\mu$ to $(1-MB_\mu(q,t))\MacH_\mu$,  
which may also be described directly (using \autoref{defDn}) without previous knowledge of the Macdonald polynomials.
Starting with $\xkn{1}{0}=D_0$ and $\xkn{0}{1}=\pi_1^\bullet$ as initial values, we recursively set
\begin{equation}\label{Def_pi_kn}
   \bleu{\xkn{k}{n}}:=\bleu{\frac{1}{(1-q)(1-t)}\, [\xkn{r'}{s'},\xkn{r}{s}]}.
\end{equation}
Here $[-,-]$ denotes the usual Lie bracket.  
  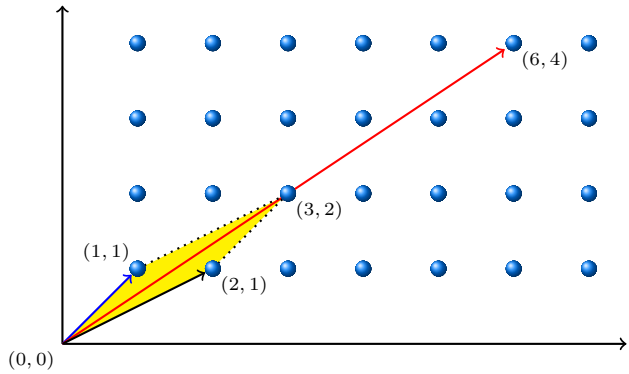
\begin{figure}[ht]
 \begin{tikzpicture}[scale=1]
\coordinate (0) at (0,0);
\coordinate (11) at (1,1);
\coordinate (32) at (3,2);
\coordinate (21) at (2,1);
\coordinate (64) at (6,4);
\draw[draw = white,fill=yellow] (0) -- (11) -- (32)--(21) --cycle;
\draw[->,thick] (0,0) -- (0,4.5);
\draw[->,thick] (0,0) -- (7.5,0);
 \draw[->,color=blue,thick] (0) -- ($(0)!.92!(11)$);
  \draw[->,color=red,thick] (0) -- ($(0)!.98!(64)$);
  \draw[color=black,thick,dotted] (11) -- (32);
 \draw[->,color=black,thick] (0) -- ($(0)!.95!(21)$);
 \draw[color=black,thick,dotted] (21) -- (32);
\foreach \x in {1,...,7}
    \foreach \y in {1,...,4}
    {
    \shade[ball color=azure]  (\x,\y) circle (3pt);
    }
 \begin{tiny}
 \draw[anchor=north west] (64) node {$(6,4)$};
 \draw[anchor=north west] (32) node {$(3,2)$};
 \draw[anchor=north west] (21) node {$(2,1)$};
  \draw[anchor=south east] (11) node {$(1,1)$};
    \draw[anchor=north east] (0) node {$(0,0)$};
\end{tiny}
\end{tikzpicture}
	\caption{Example of a splitting.}
	\label{Fig_Split}
 \end{figure}
It is easy to check that we have the above splitting if and only if, for any $k,j\in \N$, 
   \begin{equation}
     \bleu{\begin{pmatrix} 1&j\\0 &1\end{pmatrix}\begin{pmatrix} r&s\\r' &s'\end{pmatrix}}\  {\rm splits}\  \bleu{(k+jn,n)},
     \quad{\rm and}\quad
     \bleu{\begin{pmatrix} 1&0\\k &1\end{pmatrix}\begin{pmatrix} r&s\\r' &s'\end{pmatrix}}\  {\rm splits}\  \bleu{(k,n+jk)}.
  \end{equation}
This ties in nicely with the nice following property of the operators (shown to hold in~\cite{gorsky}):
\begin{equation}\label{PropConjNabla}
       \bleu{\nabla \HallOper{f}{a}{b}\nabla^{-1}}= \bleu{\HallOper{f}{a+b}{b}},
\end{equation}
involving the Macdonald eigenoperator $\nabla$ (see~\autoref{SecMacOper}), which plays a prevalent role in this subject (see~\cite{nabla,HHLRU,haimanJAC}). It is possible to  express all the $\xkn{k}{n}$ (for $1\leq k,n$) as more ``compact'' Lie bracketing expressions involving only the operators $\Dk{j}:=\xkn{1}{j}$. This results in more efficient computer algebra calculations.

\subsection{Explicit examples} Since $(1,k)$ splits as $\left(\begin{smallmatrix}1&k-1\\0 &1\end{smallmatrix}\right)$, for all $k$ we have
\begin{align}\label{X_1k}
   \bleu{\Dk{k}}&=\bleu{\xkn{1}{k}}\\
   &=\bleu{\frac{1}{M}\, [\xkn{0}{1},\xkn{1}{k-1}]}\\
   &=\bleu{\frac{1}{M^k}[\underbrace{p_1^\bullet,[p_1^\bullet,\cdots [p_1^\bullet}_{k\ {\rm copies}},\Dk{0}]\cdots ]]}.\label{X_k_crochet}
\end{align}
Similarly, we get 
\begin{align}\label{X_k1}
   \bleu{\xkn{k}{1}}&=\bleu{\frac{1}{M^k}[[\cdots [p_1^\bullet,\underbrace{\Dk{0}],\cdots \Dk{0}],\Dk{0}}_{k\ {\rm copies}}]}.
\end{align}
These expression may be represented as binary trees (see \autoref{fig_tree}) with operators $p_1^\bullet$ as left leaves, and $\Dk{0}$ as right ones.
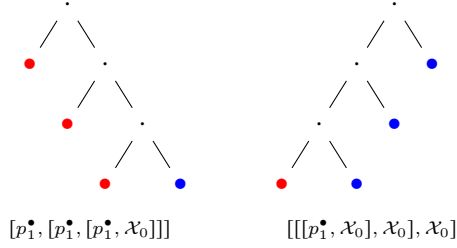
\begin{figure}
\begin{tikzpicture}[level distance=.8cm,
  level 1/.style={sibling distance=1cm},
  level 2/.style={sibling distance=1cm}]
  \node {$\cdot$}
    child {node {\rouge{$\bullet$}}}
    child {node {$\cdot$}
             child {node {\rouge{$\bullet$}}}
             child {node {$\cdot$}
             		child {node {\rouge{$\bullet$}}}
    	    		child {node {\bleu{$\bullet$}}}}};
\begin{tiny}
 \node at (0.3,-3) {$[p_1^\bullet,[p_1^\bullet,[p_1^\bullet,\Dk{0}]]]$};
 \end{tiny}
\end{tikzpicture}
\qquad
\begin{tikzpicture}[level distance=.8cm,
  level 1/.style={sibling distance=1cm},
  level 2/.style={sibling distance=1cm}]
  \node {$\cdot$}
    child {node {$\cdot$}
             child {node {$\cdot$}
             		child {node {\rouge{$\bullet$}}}
             		child {node {\bleu{$\bullet$}}}}
	    child {node {\bleu{$\bullet$}}}
    	    }
    child {node {\bleu{$\bullet$}}};
   \begin{tiny}
 \node at (-0.3,-3) {$[[[p_1^\bullet,\Dk{0}],\Dk{0}],\Dk{0}]$};
 \end{tiny}
\end{tikzpicture}
\caption{Brackettings (red nodes $=p_1^\bullet$, and blue nodes $=\Dk{0}$).}
\label{fig_tree}
\end{figure}
The binary tree expression for $\xkn{k}{n}$ and $\xkn{n}{k}$ are related to one another by a vertical reflexion symmetry, as illustrated in a special case in \autoref{fig_tree}.  Moreover, 
 \begin{equation} \label{nabla_e1}
    \bleu{\nabla p_1^\bullet\nabla^{-1}} = \bleu{\frac{1}{M}\,[p_1^\bullet,\Dk{0}] }=\bleu{\Dk{1}},
    \end{equation}
together with equations \pref{X_k_crochet} and \pref{PropConjNabla}, gives
\begin{equation}\label{nabla_D1}
    \bleu{\xkn{k+1}{k}} = \bleu{\nabla \Dk{k}\nabla^{-1}}=\bleu{[\underbrace{\Dk{1},[\Dk{1},\cdots [\Dk{1}}_{k\ {\rm copies}},\Dk{0}]\cdots ]]}.
 \end{equation}
In other terms, conjugation by the operator $\nabla$ corresponds to replacing each $p_1^\bullet$ by $\Dk{1}$. Thus, in the binary tree representation, each left leaf is replaced by the tree for $[p_i^\bullet,\Dk{0}]$: 
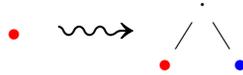
$$\begin{tikzpicture}
  \node at (-.3,1) {\rouge{$\bullet$}};
  \node at (.8,1) {\huge $\rightsquigarrow$};
  \node at (0,.5) {};
\end{tikzpicture}
\begin{tikzpicture}
[level distance=.8cm,
  level 1/.style={sibling distance=1cm},
  level 2/.style={sibling distance=.8cm}]
  \node {$\cdot$}
     child {node {\rouge{$\bullet$}}}
     child {node {\bleu{$\bullet$}}};
\end{tikzpicture}$$

\subsection*{Interesting families}
Physicists would say that the $\HallOper{f}{a}{b}$ are ``creation'' operators. Indeed, we typically apply them to the symmetric function $1$, in order to construct the symmetric functions $(\HallOper{f}{a}{b}\cdot 1)$. We need only express $f(\x)$ in the $\pi_\mu(\x)$ to get the coefficient needed to expand $(\HallOper{f}{a}{b}\cdot 1)$ in terms of $(\HallOper{\pi_\mu}{a}{b}\cdot 1)$. For sure, relations between operators translate into identities between the resulting families.  Assume that $f=f_d$ is some chosen symmetric function of degree $d$. We denote by $f_{kn}(q,t;\x)=f_{(k,n)}(q,t;\x)$ the symmetric function $(\HallOper{f_d}{a}{b}\cdot 1)$, still with the assumption that $(k,n)=(ad,bd)$ with $(a,b)$ coprime. Observe that it directly follows from \autoref{PropConjNabla}, that we have
\begin{equation}
  \bleu{\nabla (f_{(k,n)}(q,t;\x))}=\bleu{ f_{(k+n,n)}(q,t;\x)}.
\end{equation}
In particular, since $f_{(0,d)}(q,t;\x) =f_{d}(\x)$, we get 
\begin{equation}\label{Nabla_fd}
  \bleu{\nabla (f_{d}(\x))}=\bleu{ f_{(d,d)}(q,t;\x)}.
\end{equation} 
We would like to choose $f_d$, so that $f_{kn}(q,t;\x)$ is Schur positive for all $k$ and $n$. Experimental data suggest that this is the case when $f_d$ is either  $\pi_d$, $\widehat{p}_d=(-1)^{d-1}p_d$, $e_d(\x)$, or $\widehat{h}_d(\x):=(-qt)^{-d+1} h_d(\x)$ (among others).
Associated functions via the above creation procedure are respectively denoted by
\begin{equation}\label{InterestingFamilies}
   \bleu{\pi_{kn}(q,t;\x)},\qquad
   \bleu{\widehat{p}_{kn}(q,t;\x)},\qquad 
   \bleu{e_{kn}(q,t;\x)},\qquad {\rm and}\qquad 
   \bleu{\widehat{h}_{kn}(q,t;\x)}.
 \end{equation}
A further interesting family is associated to the renormalized Schur functions
$$\bleu{\widehat{s}_\mu(\x)}:=\bleu{\frac{(-1)^{\iota(\mu)}}{(qt)^{n-l(\mu)}} s_\mu(\x)},$$
where $\iota(\mu)$ stands for the number of cells $(i,j)$ in $\mu$ for which $i>j$. One of the longest-standing conjectures about  the $\nabla$ operator may be extended\footnote{The original version, stated in 1994 when $\nabla$ was introduced, was for $a=b=1$.} as follows.
\begin{conjecture}[see~\cite{BergeronOpen}] \label{ConjNablaSchur}
  For all $\mu$ and all $(a,b)$ coprime, the symmetric function $(\HallOper{\widehat{s}_\mu}{a}{b}\cdot 1)$ has a Schur expansion in the $\x$ variables, with $(q,t)$-coefficients that are $(q,t)$-Schur positive.
\end{conjecture}
Indeed, in view of \autoref{Nabla_fd}, the original conjecture about $\nabla$ (which may be found in~\cite{nabla}) corresponds to the special case $(a,b)=(1,1)$. Moreover, since $e_{kn}$ and $\widehat{h}_{kn}$ are but special instances of it, and we have the positive expansions
   $$\bleu{\pi_d=\sum_{k=1}^{n} \widehat{s}_{(k,1^{n-k})}},\qquad{\rm and}\qquad
      \bleu{\widehat{p}_d=\sum_{k=1}^{n} (qt)^{\iota(\mu)} \widehat{s}_{(k,1^{n-k})}}.
   $$
Thus, \autoref{ConjNablaSchur} implies that all the families of \pref{InterestingFamilies} are Schur positive.

\subsection{Some values} To illustrate special instances of the above conjecture, for partitions of $n\leq 4$ and encoding products $s_\lambda(q,t) s_\mu(\x)$ as $s_\lambda \otimes s_\mu$, we have the following values for 
$\nabla(\widehat{s}_\mu):=(\HallOper{\widehat{s}_\mu}{1}{1}\cdot 1)$:
\begin{small}
\begin{align*}
&\nabla(\widehat{s}_1)= 1\otimes s_{1},\\
&\nabla(\widehat{s}_2) = 1\otimes s_{11},\\
&\nabla(\widehat{s}_{11})= s_1\otimes s_{11} + 1\otimes s_{2},\\
&\nabla(\widehat{s}_3)= s_1\otimes s_{111} + 1\otimes  s_{21},\\
&\nabla(\widehat{s}_{21})=  s_2\otimes s_{111} + s_1\otimes s_{21},\\
&\nabla(\widehat{s}_{111}) = (s_{11}+s_3)\otimes s_{111} + (s_1+s_2)\otimes s_{21} + 1\otimes  s_{3},\\
&\nabla(\widehat{s}_4) = \big( s_{11} + s_{3} \big) \otimes s_{1111}+ \big( s_{1} + s_{2} \big) \otimes s_{211}+ s_{1} \otimes s_{22}+ 1 \otimes s_{31},\\
&\nabla(\widehat{s}_{31}) =\big( s_{21} + s_{4} \big) \otimes s_{1111} + \big( s_{11} + s_{2} + s_{3} \big) \otimes s_{211} + s_{2} \otimes s_{22} +s_{1} \otimes s_{31},\\
&\nabla(\widehat{s}_{22}) =s_{11} \otimes s_{1111} + s_1 \otimes s_{211} + 1 \otimes s_{31},\\
&\nabla(\widehat{s}_{211}) =\big( s_{31} + s_{5} \big) \otimes s_{1111} +\big( s_{21} + s_{3} + s_{4} \big) \otimes s_{211}+ \big( s_{11} + s_{3} \big) \otimes s_{22} + s_{2} \otimes s_{31},\\
&\nabla(\widehat{s}_{1111}) =\big( s_{31} + s_{41} + s_{6} \big) \otimes s_{1111}+\big( s_{11} + s_{21} + s_{3} + s_{31} + s_{4} + s_{5} \big) \otimes s_{211}\\
&\qquad\qquad\qquad
+\big( s_{2} + s_{21} + s_{4} \big) \otimes s_{22}
+\big( s_{1} + s_{2} + s_{3} \big) \otimes s_{31}
+1\otimes s_{4}.
\end{align*}
\end{small}

  %%%%%%%%%%%%%%%%%%%%%%%%%%%%%%%
%%%%%%%%%%%%%%%%%%%%%%%%%%%%%%%
\section{Catalan Combinatorics of triangular partitions}\label{sec_combinatorial}
 We now consider weighted enumeration of some relevant families of combinatorial objects. The simplest case corresponds to the $(q,t)$-polynomial which arises as the $0$-degree component of the superpolynomial of the $(n,n+1)$-torus knot. As it happens, Catalan numbers occur as the specialization at $q=t=1$ of this polynomial. Catalan combinatorics concerns generalizations of this original context.  For our presentation we will adopt the new perspective of ``triangular partitions'' which makes many combinatorial notions more natural.
 
 Recall that there is a wide variety of combinatorial structures enumerated by Catalan numbers (see \cite{stanley}). The context that will be useful for us is the set of partitions that are contained in the staircase shape $\delta_n=(n-1,n-2,\ldots,2,1,0)$. For instance, with $n=4$, we have the $14$ (diagram of) partitions of \autoref{Figure_Ferrers}, with $\varepsilon$ denoting the empty partition.
 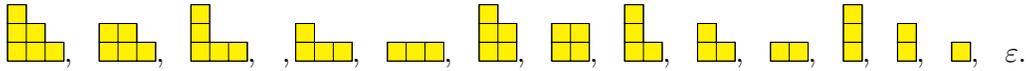
\begin{figure}[ht]
 $$\begin{array}{cccccccccccccccc} 
\begin{tikzpicture}[scale=.25] %321
\ylw
\tyng(0cm,0cm,3,2,1);
\end{tikzpicture}, 
&
\begin{tikzpicture}[scale=.25] %32
\ylw
\tyng(0cm,0cm,3,2);
\end{tikzpicture}, 
&
\begin{tikzpicture}[scale=.25] %311
\ylw
\tyng(0cm,0cm,3,1,1);
\end{tikzpicture}, 
&, 
\begin{tikzpicture}[scale=.25] %31
\ylw
\tyng(0cm,0cm,3,1);
\end{tikzpicture}, 
&
\begin{tikzpicture}[scale=.25] %3
\ylw
\tyng(0cm,0cm,3);
\end{tikzpicture}, 
&
\begin{tikzpicture}[scale=.25] %221
\ylw
\tyng(0cm,0cm,2,2,1);
\end{tikzpicture}, 
&  
\begin{tikzpicture}[scale=.25] %22
\ylw
\tyng(0cm,0cm,2,2);
\end{tikzpicture}, 
&
\begin{tikzpicture}[scale=.25] %211
\ylw
\tyng(0cm,0cm,2,1,1);
\end{tikzpicture},
&
\begin{tikzpicture}[scale=.25] %21
\ylw
\tyng(0cm,0cm,2,1);
\end{tikzpicture},
&
\begin{tikzpicture}[scale=.25] %2
\ylw
\tyng(0cm,0cm,2);
\end{tikzpicture},
&
\begin{tikzpicture}[scale=.25] %111
\ylw
\tyng(0cm,0cm,1,1,1);
\end{tikzpicture}, 
&
\begin{tikzpicture}[scale=.25] %11
\ylw
\tyng(0cm,0cm,1,1);
\end{tikzpicture}, 
&
\begin{tikzpicture}[scale=.25] %1
\ylw
\tyng(0cm,0cm,1);
\end{tikzpicture}, 
&
\varepsilon.
\end{array}$$
\caption{Partitions contained in $321$.}
\label{Figure_Ferrers}
\end{figure}
More generally, we consider the set $\mathcal{D}_{\tau}$ of partitions contained in a given ``triangular partition'' $\tau$. This set of partitions is then enumerated with respect to two statistics, giving the polynomial
\begin{equation}
   \bleu{\mathcal{D}_{\tau}(q,t)}:=\bleu{\sum_{\mu\subseteq \tau} q^{\area_\tau(\mu)} t^{\mathrm{sim}_\tau(\mu)}}.
\end{equation}
Here \define{$\area_\tau(\mu)$} is simply the number of cells of $\tau$ which do not lie in $\mu$. The second statistic\footnote{Previously called ``dinv'' in the literature.}, called $\mathrm{sim}(\mu)$ and which is described below, measures how ``similar'' $\mu$ is to $\tau$.

A partition $\tau=\tau_1\tau_2\cdots \tau_k$ is said to be \define{triangular} if there exists positive real numbers $r$ and $s$ such that 
$$\bleu{\tau_j} =  \bleu{\big\lfloor r-{j\,r}/{s}\big\rfloor},$$
with $j$ running over integers that are less or equal to $s$. For sure some of the $\tau_j$ may vanish, and such $0$ parts are usually removed. In other words, as illustrated in \autoref{fig_triangular}, the cells of the partition $\tau$ are those that lie entirely below the \define{diagonal} line joining
$(0,s)$ and $(r,0)$. 
\autoref{TableTriangularPartitions} displays all triangular partitions of $m\leq 6$.
\begin{table}[ht]
\ylw
 \Yboxdim{.2cm}
\begin{align*}
&0,\quad \yng(1),\quad \yng(2),\yng(1,1),\quad \yng(3),\yng(2,1),\yng(1,1,1),\quad \yng(4),\yng(3,1),\yng(2,1,1),\yng(1,1,1,1),\\
&\yng(5),\yng(4,1),\yng(3,2),\yng(2,2,1),\yng(2,1,1,1),\yng(1,1,1,1,1),\quad \yng(6),\yng(5,1),\yng(4,2),\yng(3,2,1),\yng(2,2,1,1),\yng(2,1,1,1,1),\yng(1,1,1,1,1,1).
\end{align*}
\caption{All triangular partitions, for $n\leq 6$.}\label{TableTriangularPartitions}
\end{table}

An (positive coordinates) orthogonal vector to this line is said to be a \define{slope vector} of the partition. We may assume this vector is normalized so that it is of the form $(t,1-t)$, for $0\leq t\leq 1$.
\begin{figure}[ht]
\begin{tikzpicture}[scale=.5]
\draw[blue,opacity=.4] (0,6.5) --(9.8,0);
\draw[step=1,gray,thin] (0,0) grid (10,7);
\tyng(0cm,0cm,8,6,5,3,2);
\node at (-1,6.5) {$s$};
\filldraw[black] (0,6.5) circle (3pt);
\node at (10,-.8) {$r$};
\filldraw[black] (9.8,0) circle (3pt);
\draw[->,color=red,line width=.7mm] (4.5,3.5)-- ($(4.5,3.5)+(6.5/7,9.8/7)$);
\node at ($(4.5,3.5)+(8/7,11/7)$) {$\overline{v}$};
\end{tikzpicture}
\caption{A triangular partition, and slope vector.}
\label{fig_triangular}
\end{figure}
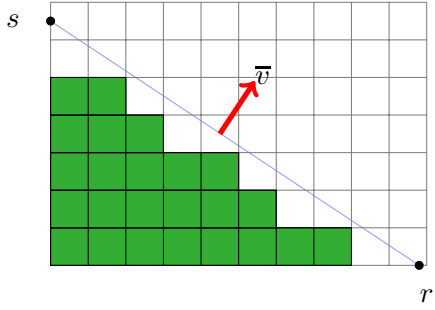

For any hook shape $(a\,|\,\ell):=(a+1,1^\ell)$ one considers the two slope vectors $v'=(t',1-t')$ and $v''=(t'',1-t'')$ (as illustrated  in \autoref{fig_hook_slopes})  of the two lines: 
\begin{itemize}
\item $d'$ which joins the vertices $(1,\ell+1)$ and $(a+2,1)$, and 
\item $d''$ which joins the vertices $(1,\ell+2)$ and $(a+1,1)$.
\end{itemize}
 It is easy to check that
\begin{equation}\label{def_slope_t}
\bleu{t'}= \bleu{\frac{\ell}{a+\ell+1}},\qquad {\rm and}\qquad \bleu{t''} = \bleu{\frac{\ell+1}{a+\ell+1}}.
\end{equation}
 Clearly, for any $t'<t<t''$, the vector $(t,1-t)$ lies in the convex cone having as bounding rays positive multiples of the vectors $v'$ and $v''$. Informally, vectors in this cone may be considered as slope vectors for hypothenuse of discrete approximations of the "right triangle" having height $(\ell+1)$ and basis $(a+1)$. We may thus consider that $(t',t'')$ is the slope interval of the hook. Moreover, two hooks are said to be \define{similar} if the intersection of their slope interval is not empty. In other words, they share a common slope vector.
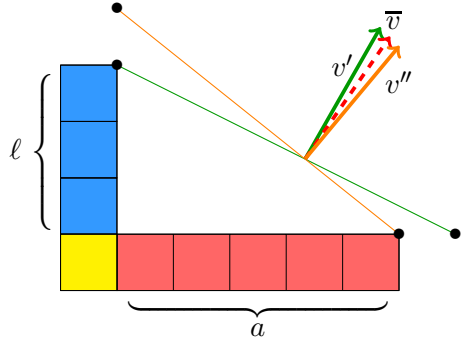
\begin{figure}[ht]
 \begin{tikzpicture}[scale=.75]
    \Yfillcolour{azure!80}
   \tyng(0cm,1cm,1,1,1);
\Yfillcolour{red!60}
   \tyng(1cm,0cm,5);
\Yfillcolour{yellow}
   \tyng(0cm,0cm,1);
     \node at (-.5,2.5) {$\ell\left.\rule{0cm}{34pt}\right\{$};
 \node at (3.5,-0.5) {$\underbrace{\hskip3.5cm}_{\textstyle a}$};
  \draw[color=green] (1,4) -- (7,1);
 \draw[color=orange] (1,5) -- (6,1);
 \node at (1,5) {$\bullet$};
 \node  at (1,4) {$\bullet$};
 \node at (6,1) {$\bullet$};
 \node at (7,1) {$\bullet$};
 \draw[->,color=green,line width=.5mm] ($(4.34,2.34)+(0,0)$)--($(4.34,2.34)+(4/3,7/3)$);
 \draw[->,dashed,color=red,line width=.5mm] ($(4.34,2.34)+(0,0)$)--($(4.34,2.34)+(9/6,13/6)$);
 \draw[->,color=orange,line width=.5mm] ($(4.34,2.34)+(0,0)$)--($(4.34,2.34)+(5/3,6/3)$);
 \node at (5,4) {$v'$};
  \node at (5.9,4.8) {$\overline{v}$};
  \node at (6,3.7) {$v''$};
  \end{tikzpicture}
  \caption{The two extreme slope vectors (here upscaled) of a hook shape $(a\,|\,\ell)$.}
  \label{fig_hook_slopes}
  \end{figure}

When we apply this construction to a cell $c$ of a partition $\mu$, we denote by $t'(c,\mu)$ and $t''(c,\mu)$ the quantities given by \autoref{def_slope_t}, with $a$ and $\ell$ respectively standing for the arm and leg of $c$ in $\mu$. The \define{slope cone} $\mathrm{SC}(\mu) $ of a partition $\mu$ is the (possibly empty) set 
\begin{equation}\label{def_t_mu}
   \bleu{\mathrm{SC}(\mu)} := \bleu{\{t\,e_1+(1-t)\,e_2 \ |\ t^{-}_\mu < t< t^{+}_\mu\}},
\end{equation}
with 
\begin{equation}\label{def_t_mu}
 \bleu{t^{-}_\mu} = \bleu{\max_{c\in\mu}\ t'(c,\mu)},\qquad {\rm and}\qquad 
\bleu{t^{+}_\mu} = \bleu{\min_{c\in\mu}\ t''(c,\mu)}.
\end{equation}
One may show that a partition is triangular if and only if $\mathrm{SC}(\mu)$ is not empty. Equivalently, $\tau$ is triangular if and only if the (open) \define{slope interval} $\mathrm{sc}(\tau):=(t^{-}_\tau,t^{+}_\tau)$ is not empty. In other words, all the hooks of cells in $\tau$ are similar. 
For example for the three cells of $\tau=21$, we calculate that 
\begin{table}[ht]\renewcommand{\arraystretch}{1.5}
\begin{tabular}{|c||c|c|c|c|}
\hline
\rowcolor{yellow} cell&(0,0)&(1,0)&(0,1)\\
\hline\hline
\cellcolor{yellow}$\frac{\ell}{a+\ell+1}$ &1/3&0&0\\ 
\hline
\cellcolor{yellow}$\frac{\ell+1}{a+\ell+1}$ &2/3&1&1\\ 
\hline
\end{tabular}
\medskip
\caption{Rectangular Catalan numbers $\Cat_{rs}$.}\label{tab2}
\end{table}
Hence, the slope interval of this partition is $(1/3,2/3)$. 

Now, consider a subpartition $\mu$ of a given triangular partition $\tau$. Consider the set of cells $c$ of $\mu$ that have hooks in $\mu$ that are ``similar'' to $\tau$:
\begin{equation}
   \bleu{\Sim_\tau(\mu)}:=\bleu{\{ c\in \mu\ |\ t'(c,\mu)<\overline{t}<t''(c,\mu)) \}},\qquad {\rm where}\qquad \overline{t}=\frac{t^{-}_\tau+t^{+}_\tau}{2}.
\end{equation}
The cardinality of this set, denoted by \define{$\simd_\tau(\mu)$}, measures the level of \define{similarity} between $\mu$ and $\tau$. Expressed in words, the similarity of $\mu$ with $\tau$ counts the number of cells of $\mu$ having ``hook triangle slope vectors''  compatible with the (average) slope vector of $\tau$ (represented by $\overline{v}$ in \autoref{fig_hook_slopes}). The following is an example of a sequence of similar triangular partitions:
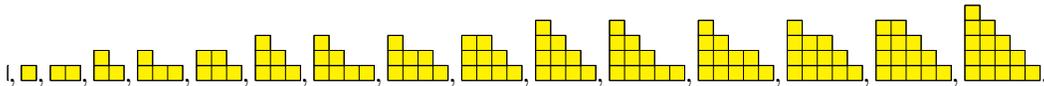
\begin{center}
\ylw
 \Yboxdim{.2cm}
\begin{align*}
&\yng(0),\yng(1), \yng(2),\yng(2,1), \yng(3,1),\yng(3,2),\yng(3,2,1), \yng(4,2,1),\yng(4,3,1),\yng(4,3,2),\yng(4,3,2,1),
\yng(5,3,2,1),\yng(5,4,2,1),\yng(5,4,3,1),\yng(5,4,3,2),\yng(5,4,3,2,1).
\end{align*}
\end{center}
They all share the common slope vector $\frac{1}{2}(1-\epsilon,1+\epsilon)$.

In particular $\mu$ is entirely similar to $\tau$ if and only if all its cells lie in $\simd_\tau(\mu)$. Then, both partitions are triangular, and they share a common slope vector (in fact many). It may be shown that there is exactly one such $\mu$ of size $k$, for all $0\leq k\leq |\tau|$. It follows that, for any triangular partition $\tau$ of size $N$, we have
\begin{align}
   \bleu{\mathcal{D}_{\tau}(q,t)}& = \bleu{(q^N+q^{N-1}t+\ldots + q t^{N-1}+t^N)+\ldots },\nonumber\\
       &=\bleu{s_N(q,t)+\ldots} \label{observation_Dn}
\end{align}
where the remaining terms are of degree (strictly) less than $n$. In the particular case $\delta_n=(n,n-1,\ldots,2,1,0)$, one may further see that
  \begin{equation}
   \bleu{\mathcal{D}_{\delta_n}(q,t)}=\bleu{\ldots +\Big(\sum_{k=\binom{n}{2}}^{\binom{n+1}{2}-2} s_{k,1}\Big) + s_{\binom{n+1}{2}}},
 \end{equation}
 where the missing terms are indexed by partitions having two parts, the second of which being larger or equal to $2$. It is also interesting to observe that, for $a+\ell=n-1$, we have
 	\begin{equation}
	   \bleu{\mathcal{D}_{\delta_n+1^{\ell}}(q,t;\a)}=\bleu{\frac{1}{1+\a} \nabla(\widehat{s}_{(a\,|\, \ell)})[1-\varepsilon \a])}.
	 \end{equation}

 \subsection{Examples of small values} Taking into account that $\mathcal{D}_{\tau}=\mathcal{D}_{\tau'}$, the values of $\mathcal{D}_{\tau}$ for (all) triangular partitions of size at most $8$ are given in \autoref{table_D_tau}.  \begin{table}[ht]
\begin{tabular}{lllllllllllllllllll}
$\left(0, \boldsymbol{1}\right)$ \\
$\left(1, s_{1}\right)$ \\
$\left(2, s_{2}\right)$ \\
$\left(21, s_{11} + s_{3}\right)$ & $\left(3, s_{3}\right)$ \\
$\left(31, s_{21} + s_{4}\right)$ & $\left(4, s_{4}\right)$ \\
$\left(32, s_{31} + s_{5}\right)$ & $\left(41, s_{31} + s_{5}\right)$ & $\left(5, s_{5}\right)$ \\
$\left(321, s_{31} + s_{41} + s_{6}\right)$ & $\left(42, s_{22} + s_{41} + s_{6}\right)$ & $\left(51, s_{41} + s_{6}\right)$ & $\left(6, s_{6}\right)$ \\
$\left(421, s_{32} + s_{41} + s_{51} + s_{7}\right)$ & $\left(52, s_{32} + s_{51} + s_{7}\right)$ & $\left(61, s_{51} + s_{7}\right)$ & $\left(7, s_{7}\right)$\\
$\left(431, s_{42} + s_{51} + s_{61} + s_{8}\right)$ & $\left(53, s_{42} + s_{61} + s_{8}\right)$ & $\left(62, s_{42} + s_{61} + s_{8}\right)$ & $\left(71, s_{61} + s_{8}\right)$ & $\left(8, s_{8}\right)$
\end{tabular}
\caption{Table of values $(\tau,\mathcal{D}_{\tau})$.}
\label{table_D_tau}
\end{table}

%%%%%%%%%%%%%%%%%%%%%%%%%%%%%%%
\subsection{The combinatorial side of the Delta Conjecture}\label{delta_comb}
We say that $1\leq i< l(\mu)$ is a \define{descent} of a partition $\mu$ if $\mu_i>\mu_{i+1}$, and we denote by $\des(\mu)$ the set of descent of $\mu$.  Let $\mu$ be contained in the triangular partition $\tau=(n,n-1,\ldots,1,0)$.
For a subset $J$ of $[n]:=\{1,\cdots,n\}$, we consider the number of $\tau$-area cells for $\mu$ which lie on rows lying in the set $J$, {\sl i.e.}:
    $$\bleu{\alpha_\tau(\mu, J)}:=\bleu{\sum_{j\in J} (\tau_j-\mu_j)}.$$
 With these notions at hand, we may define the polynomial
\begin{equation}\label{delta_conjecture}
  \bleu{\mathbb{D}_{\tau}(q,t;\a)}:=\bleu{\sum_{\mu\subseteq \tau} t^{\simd_{\tau}(\mu)} \sum_{\des(\mu)\subseteq J }q^{\alpha_\tau(\mu, J)} \a^{n-|J|}},
  \end{equation}
  which occurs on the combinatorial side of the ``Delta conjecture'' (see \cite{HaglundKnot} for more on this). The coefficient of $\a^0$ in $\mathbb{D}_{\tau}(q,t;\a)$ is clearly the polynomial $\mathcal{D}_{\tau}(q,t)$, since it corresponds to fixing $J=[n]$.   At the other extreme, $J=\varnothing$ forces $\des(\mu)=\varnothing$, hence $\mu=\varepsilon$. It follows that the coefficient of $\a^n$ is equal to $1$. One may show that the coefficient of $\a^{n-1}$ is the $(q,t)$-symmetric function $s_1+s_2+\ldots+s_n$. 
  To summarize,
  \begin{equation}
  \bleu{\mathbb{D}_{\tau}(q,t;\a)}:=\bleu{\mathcal{D}_{\tau}(q,t)\,\a^0 + \ldots +(s_1+s_2+\ldots+s_n)\,\a^{n-1}+\a^n}.
  \end{equation}

%%%%%%%%%%%%%%%%%%%%%%%%%%%%%%%
%%%%%%%%%%%%%%%%%%%%%%%%%%%%%%%
\section{Link to Link homology}\label{sec_homology}
In the context of link homology, the superpolynomial (Poincaré polynomial of the triply graded Khovanov-Rozansky homology) arises as (see \autoref{eval_hook_1a})
\begin{align}
  \bleu{ \mathcal{P}_{kn}(q,t;\a)} &= \bleu{\frac{e_{kn}[1-\varepsilon \a]}{1+\a}}\nonumber \\
                                                    &= \bleu{\sum_{i+j=n-1}\langle e_{kn}(\x),s_{(i\,|\,j)}\rangle\, \a^i}.
\end{align}
where $(k,n)=(ad,bd)$, with $(a,b)$ coprime. Recall that $e_{kn}(\x)=e_{kn}(q,t;\x):=\HallOper{e_d}{a}{b}(1)$.  For a torus knot, $d=1$. 
Although $e_{kn}(\x)\neq e_{nk}(\x)$ in general, both symmetric functions (appear to) agree on their hook-components. This is forced if the above is to hold, since the $(k,n)$ torus link coincides with the $(k,n)$ torus link. For the purpose of calculating $\mathcal{P}_{kn}$ via $e_{kn}(q,t;\x)$, we may choose $k\geq n$. Note that, for $k\geq n$, we have the following evaluation
  \begin{align}
  	\bleu{e_{kn}[q,0;1-u]} &= \bleu{q^{\delta_{kn}}\MacH_n(\x)[1-u]}\nonumber\\
	   &=\bleu{q^{\delta_{kn}}\prod_{i=0}^{n-1} (1-q^iu)},\label{eval_q0} 
  \end{align}
  where $\delta_{kn}:=\sum_{j}\big\lfloor k(n-j)/{n}\big\rfloor-j-1.$

The superpolynomial is symmetric in $q$ and $t$, and in fact it is Schur positive. This is to say that the coefficient of each $A^i$ is a linear combination of Schur functions in $q,t$, with its coefficients in $\mathbb{N}$. For example, for $k=2r+1$ and $n=2$, we have the general formulas:
\begin{align}
   \bleu{\mathcal{P}_{2r+1,2}(q,t;\a)} &=\bleu{ s_r+  s_{r-1} \,\a},\label{cas_deux}\\
   \bleu{\mathcal{P}_{3r+1,3}(q,t;\a)}  &= \bleu{\rho_{r}^{r}+(\rho_{r}^{r-1}+\rho_{r+1}^{r-1})\,\a+\rho_{r-1}^{r-1}\,\a^2}; \label{cas_trois}       
\end{align}
with $\rho_{r}^{k}$ standing for $\sum_{j=0}^k s_{r-2j,k+j}$. Other small values of $\mathcal{P}_{kn}=\mathcal{P}_{kn}(q,t;\a)$ are:
\begin{align*}
\mathcal{P}_{54} &=  (s_{31} + s_{41} + s_{6}) + (s_{11} + s_{21} + s_{31}  + s_{3} + s_{4} + s_{5})\,\a + (s_{1} + s_{2} + s_{3})\,\a^{2}+ \a^{3} ,\\
\mathcal{P}_{65} &= 
(s_{43} + s_{42} + s_{62}  + s_{61} + s_{71} + s_{81} + s_{(10)})\\
&\qquad+ (s_{33} + s_{32}+ s_{42} + s_{52}  + s_{31} +   2s_{41} + 2s_{51}  + 2s_{61} + s_{71}  + s_{6} + s_{7} + s_{8} + s_{9})\,\a\\
&\qquad+  ( s_{32} + s_{11} + s_{21} + 2s_{31} + s_{41} + s_{51} + s_{3}  + s_{4}  + 2s_{5} + s_{6} + s_{7})\,\a^{2}\\
&\qquad+ (s_{1} + s_{2} + s_{3} + s_{4})\,\a^{3}+ \a^{4}.
\end{align*}
Both of these are special cases of the more general identity (see \autoref{delta_comb}):
\begin{equation}
  \bleu{\mathcal{P}_{n+1,n}(q,t;\a)} = \bleu{\mathbb{D}_{\delta_{n}}(q,t;\a)},
\end{equation}
with $\delta_n=(n-1,\ldots,1,0)$.
Observe in the above examples that we have the Schur positivity:
\begin{equation}
   \bleu{(\mathcal{P}_{kn}\big|_{\a^i}) -e_i^\perp (\mathcal{P}_{kn}\big|_{\a^0})\ \in\ \N[s_\mu\, |\,\mu\ {\mathrm{partition}}]},
 \end{equation}
 for all $i$ (although this is trivial for $i>2$).
Here $(-)\big|_{\a^i}$ means that we take the coefficient of $\a^i$.

 We conjecture that, for all $(k,n)$, there exists\footnote{It is unique if we require that it be ``smallest'' possible. In some known instances $\A_{kn}$ arises as the multiplicity enumerator of alternating components for certain $\S_n$-modules (see \cite{BergeronGlkSnMult}).} a Schur positive polynomial $\A_{kn}$ such that
\begin{equation}\label{cond_skew}
  \bleu{\forall (i\geq 0)\quad (e_i^\perp \A_{kn})(q,t)} =\bleu{\mathcal{P}_{kn}\big|_{\a^i}} = \bleu{\langle e_{kn}(\x),s_{(i\,|\,j)}\rangle}.
 \end{equation}
For the case $i=0$ to hold, we clearly need $\A_{kn}$ to coincide with $\mathcal{P}_{kn}$ on terms involving Schur function indexed by partitions of length at most two. In general, the $\A_{kn}$ will also involve some extra Schur functions, indexed by partitions having more than two parts.  

As first explicit examples, we may directly check that setting $\A_{2r+1,2}=s_r$ and $\A_{3r+1,3}=\rho_r^r$ do satisfy the required conditions for all $r$, with the values given in \autoref{cas_deux} and \autoref{cas_trois}. 
Further examples of explicit values are:
\begin{align*}
  &\A_{54} = s_{111} + s_{31} + s_{41} + s_{6},\\ 
  &\A_{65} = \underbrace{s_{1111} + s_{311} + s_{411} + s_{511}}_{\mathrm{extra\ terms}}  + s_{42} + s_{43} + s_{61} + s_{62} + s_{71} + s_{81} + s_{(10)}.
\end{align*}
In this last case, considering the coefficients of $\a^4$ and $\a^3$ in $\mathcal{P}_{65} $ given above, we see that the extra terms $s_{1111} + s_{311} + s_{411} + s_{511}$ are minimally required if we are to get
\begin{itemize}
\item $1$ by applying $e_4^\perp$ to $\A_{65}$, and
\item $s_1+s_2+s_3+s_4$ by applying $e_3^\perp$ to $\A_{65}$.
\end{itemize}
We may then check that condition \pref{cond_skew} is satisfied  for ``free'' at $i=1$. Observe that, using \autoref{hook_eval_1q} and  \autoref{eval_q0}, it follows from Conjecture \pref{cond_skew} that
\begin{equation}\label{formule_hook}
    \bleu{\frac{\A_{kn}[y(1-z)]}{y(1-z)}} = \bleu{y^{\delta_{kn}}(1-z)(y-z)\cdots (y^{n-2}-z)}
\end{equation}
This characterizes all terms indexed by hooks in $\A_{kn}$, since 
	$$\frac{s_{(a\,|\,\ell)}[y(1-z)]}{y(1-z)} = y^{a+\ell} (-z)^\ell,$$ 
and $s_\mu[y(1-z)]=0$ when $\mu$ is not a hook. We can read off the hook terms from the monomials that occur in the right-hand side of \autoref{formule_hook}, and these number $2^{n-2}$.

%%%%%%%%%%%%%%%%%%%%%%%%%%%%%%%
%%%%%%%%          Appendixes          %%%%%%%%%%%
%%%%%%%%%%%%%%%%%%%%%%%%%%%%%%%
\null\vfill\break
\appendix

%%%%%%%%%%%%%%%%%%%%%%%%%%%%%%%%%%%%%%%%%%%%
\section{Partition terminology}\label{sec_part}

\begin{wrapfigure}{r}{0.25\textwidth}
\vskip-10pt
\begin{tikzpicture}[scale=.3,>=stealth',auto,node distance=3cm, thick]
\coordinate (1) at (12,5.5);
\coordinate (2) at (8.5,4.2);
\coordinate (3) at (10,8.5);
\coordinate (4) at (6.2,6.5);
\coordinate (11) at (1.5,5.5);
\coordinate (12) at (3.5,4.2);
\coordinate (13) at (1.5,1.5);
\coordinate (14) at (4.8,1.5);
\ylw
\tyng(2cm,0cm,13,12,12,11,9,9,8,7,7,5,1);
\Yfillcolour{green!50}
\tyng(5cm,3cm,1);
\Yfillcolour{red!50}
\tyng(6cm,3cm,7);
\Yfillcolour{azure!80}
\tyng(5cm,4cm,1,1,1,1,1,1);
\draw [->,red] (1.west) to [out=180,in=60] (2.north);
\draw [->,azure] (3.west) to [out=180,in=0] (4.east);
\begin{scriptsize}
\node at ($(1)+(1.2,0)$) {$a(c)$};
\node at ($(3)+(1.2,0)$) {$\ell(c)$};
\node at (5.4,3.5) {$c$};
\end{scriptsize}
\end{tikzpicture}
\vskip-5pt
\wrapcaption{Arm and leg}
\vskip-12pt
\end{wrapfigure}
%%%%%%%%%%%%%%%%
As usual \define{partitions} are decreasing sequences $\mu=\mu_1\mu_2\cdots \mu_\ell$ of integers $\mu_i\geq 0$. We denote by $0$, the unique \define{empty} partition. The $\mu_i$'s are the \define{parts} of $\mu$, and $|\mu|:=\mu_1+\mu_2+\ldots+\mu_\ell$ is its \define{size}. When $|\mu|=m$, $\mu$ is said to be a partition of $m$, and this is denoted by $\mu\vdash m$.  The \define{length} of $\mu$, denoted by $l(\mu)$, is the number of its non-zero parts.  One denotes by \define{$\eta(\mu)$} the integer $\sum_{j} (j-1)\mu_j$. The \define{(Ferrers) diagram}\footnote{Naturally drawn using Cartesian coordinates.}  of a partition $\mu$ is the set of \define{cells} $(i,j)$, for which $1\leq j\leq k$, and 
$1\leq i\leq \mu_j$. Most often, we denote the same way a partition and its diagram.  The cell $c=(i,j)$ is said to sit in \define{column}\index{partition!column} $i$ and \define{row}\index{partition!row} $j$ of (the diagram of) $\mu$.  Observe that cells have cartesian coordinates.

The \define{arm}  of a cell $c=(i,j)$ of $\mu$ is the number $a(c)=a_\mu(c):=\mu_j-i$ of cells of $\mu$ that sit to the right of $c$ on the same row. The \define{leg} of $c$ is the number $\ell(c)=\ell_\mu(c):=\mu'_i-j$ of cells that sit above $c$ in the same column. The \define{conjugate}\index{partition!conjugate} $\mu'$, of a partition $\mu$, is the partition whose diagram is $\mu'=\{(j,i)\,|\, (i,j)\in\mu\}$. 
Considering cells as ``vectors'' in $\N\times \N$, we have $(\eta(\mu),\eta(\mu'))=\sum_{(i,j)\in\mu} (i,j)$.
Partitions are also often specified via the \define{multiplicities} of their parts, then writing $\mu=1^{d_1}2^{d_i}\cdots n^{d_n}$ if $\mu$ contains $d_i$ parts of size $i$ (omitting $0$ multiplicities). 

The \define{hook length}\index{cell!hook length},  of a cell $c$ in $\mu$, is \define{$\hook(c)=\hook_\mu(c):= 1+a_\mu(c)+\ell_\mu(c)$}. 
In \define{Frobenius' notation for hook shape partitions}, we write $(a\, |\, \ell)$ for the partition with one part of size $a+1$, and $\ell$  parts of size $1$. 
We also denote by $\overline{\mu}$ the partition obtained from $\mu=\mu_1\mu_2\cdots \mu_k$ by removing its first part, {\sl i.e.}  $\overline{\mu}=\mu_2\cdots \mu_k$. Finally, the \define{sum} of two partitions is ${\mu+\lambda}:={\rho_1\rho_2\cdots \rho_n}$, with ${\rho_i:=\mu_i+\lambda_i}$,
 for $n=\max(l(\mu),l(\lambda))$, and partitions being padded with zeros for this to make sense.
 
Finally, recall that ${\lambda \prec \mu}$ stands for 
$\lambda$ being (strictly) smaller than $\mu$ in the \define{dominance order}, which is defined by 
          $$\bleu{\lambda\preceq \mu}\qquad  \hbox{ iff \qquad}\qquad \forall i\qquad \bleu{\sum_{j=1}^{{i}}  \lambda_j \leq \sum_{j=1}^{{i}}  \mu_j },$$
adding $0$ parts if necessary to make sense of the above. This is a partial order on partitions of $n$. For example, with $n=6$, we have 
\begin{figure}[ht]
\begin{tikzpicture}[scale=.4]
\ylw
 \Yboxdim{.2cm}
\coordinate (P6) at (-1,0);
\coordinate (P51) at (1,0);
\coordinate (P42) at (5,0);
\coordinate (P411) at (9,-2.2);
\coordinate (P33) at (9,2.2);
\coordinate (P321) at (13.2,0);
\coordinate (P3111) at (17,-2);
\coordinate (P222) at (17,2);
\coordinate (P2^211) at (22,0);
\coordinate (P21111) at (26.5,0);
\coordinate (P111111) at (32,0);
\draw[left]  (P6) node {\hbox{\yng(1,1,1,1,1,1)}};
\draw[right]  (P51) node {\hbox{\yng(2,1,1,1,1)}};
\draw[right]  (P42) node {\hbox{\yng(2,2,1,1)}};
\draw[below right,anchor=west]  (P411) node {\hbox{\yng(3,1,1,1)}};
\draw[above right,anchor=west]  (P33) node {\hbox{\yng(2,2,2)}};
\draw[right,anchor=west]  (P321) node {\hbox{\yng(3,2,1)}};
\draw[below right,anchor=west]  (P3111) node {\hbox{\yng(4,1,1)}};
\draw[above right,anchor=west]  (P222) node {\hbox{\yng(3,3)}};
\draw[right]  (P2^211)node {\hbox{\yng(4,2)}};
\draw[right]  (P21111) node {\hbox{\yng(5,1)}} ;
\draw[right] (P111111) node {\hbox{\yng(6)}};
\draw[directed,color=black,thin] (P6) to (P51);
\draw[directed,color=black,thin] ($(P51)+(2.2,0)$) to (P42);
\draw[directed,color=black,thin] ($(P42)+(2,-1)$) to  (P411);
\draw[directed,color=black,thin] ($(P42)+(2,1)$) to (P33);
\draw[directed,color=black,thin] ($(P411)+(2.2,-0)$) to ($(P321)+(0,-1)$);
\draw[directed,color=black,thin] ($(P33)+(2.2,0)$) to ($(P321)+(0,1)$);
\draw[directed,color=black,thin] ($(P321)+(1.5,1)$) to  (P222);
\draw[directed,color=black,thin] ($(P321)+(1.5,-1)$) to (P3111);
\draw[directed,color=black,thin] ($(P3111)+(2.5,0)$) to ($(P2^211)-(0,.5)$);
\draw[directed,color=black,thin] ($(P222)+(2.6,0)$) to ($(P2^211)-(0,-.5)$);
\draw[directed,color=black,thin] ($(P2^211)+(2.7,0)$) to (P21111);
\draw[directed,color=black,thin] ($(P21111)+(3.5,0)$) to (P111111);
\end{tikzpicture}
\caption{The dominance order on partitions of $6$.}\label{FigDominanceOrder}
\end{figure}
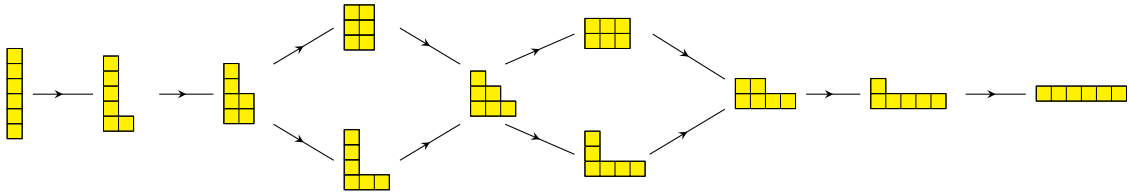

\section{Crash Course in Symmetric functions}\label{sec_sym}
To simplify all formulas, it is best to assume that  we work with a denumerable set of variables $\x=x_1,x_2,x_3,\ldots$. The ring $\Lambda_{\x}$ of symmetric polynomials\footnote{We say functions to underline that we are actually working with ``polynomials'' in infinitely many variables. See~\cite{macdonald} for a more formal definition.} in the variables $\x$ is graded by degree:
    \begin{equation}
        \bleu{\Lambda_{\x}=\bigoplus_d\, \Lambda_{\x}^{(d)}},
    \end{equation}   
with $\Lambda^{(d)}=\Lambda_{\x}^{(d)}$ denoting the degree $d$ homogeneous component. It is easy to see that $\Lambda^{(d)}$ affords as basis the partition indexed set of \define{monomial} symmetric polynomials 
\begin{equation}
   \bleu{m_\lambda(\x)}:=\bleu{\sum_{\boldsymbol{a} \in \N_{\infty}} \x^{\boldsymbol{a}}},
  \end{equation}
  in which $\N_{\infty}$ stands for finite support sequences $\boldsymbol{a}=(a_1,a_2,\ldots)$ of integers. In particular $|\boldsymbol{a} |:=\sum_i a_i$ is finite, since all but a finite number of the $a_i$'s are equal to $0$. We are here using the ``vector notation'' for monomials $ \bleu{\x^{\boldsymbol{a}}}:=\bleu{\prod_i x_i^{a_i}}$. 
The {degree} of the monomial $\x^{\boldsymbol{a}}$ is thus $|\boldsymbol{a} |$.
We mostly use the notation of~\cite{macdonald}, so that 
\begin{align*}
&\bleu{e_k(\x)}:=\bleu{m_{1^k}(\x)},&&\bleu{p_k(\x)}:=\bleu{m_k(\x)},&&\bleu{h_k(\x)}:=\bleu{\sum_{\lambda\vdash k} m_{\lambda}(\x)},
\end{align*}
 respectively  stand for the  \define{elementary}, \define{power sum}, and \define{complete homogeneous} symmetric polynomials.
When dealing with symmetric functions, it is often practical to drop the variables. Thus, for $\lambda=\lambda_1\cdots\lambda_k$ one sets
\begin{align*}
&\bleu{e_\lambda}:=\bleu{e_{\lambda_1}\cdots e_{\lambda_k}},&&\bleu{p_\lambda}:=\bleu{p_{\lambda_1}\cdots p_{\lambda_k}},&&\bleu{h_\lambda}:=\bleu{h_{\lambda_1}\cdots h_{\lambda_k}}.
\end{align*}
Each of the three sets $\{h_\lambda\}_{\lambda\vdash d}$, $\{e_\lambda\}_{\lambda\vdash d}$, and $\{p_\lambda\}_{\lambda\vdash d}$ constitute a linear basis of $\Lambda^{(d)}$. One of the most important basis for $\Lambda_{\x}^{(d)}$ is the set $\{s_\lambda(\x)\}_{\lambda\vdash d}$, of \define{Schur functions}.  These are linked to the $h_\lambda$ (and $e_\lambda$) via the \define{Jacobi-Trudi} formulas:
    \begin{align}\label{formule_jacobi}
       &\bleu{s_\mu}=\bleu{\det\left( h_{\mu_i+(j-i)}\right)_{1\leq i,j\leq l(\mu)}},\qquad {\rm and} 
      &&\bleu{s_{\mu'}}=\bleu{\det\left( e_{\mu_i+(j-i)}\right)_{1\leq i,j\leq l(\mu)}}.    
    \end{align}
We also consider, the linear and multiplicative \define{$\omega$ involution} such that $\omega(p_n)=(-1)^{(n-1)}p_n$ and $\omega(s_\mu)=s_{\mu'}$. In particular, this Hall scalar product preserving involution sends the $\{h_\lambda\}_{\lambda\vdash d}$ basis to the $\{e_\lambda\}_{\lambda\vdash d}$ basis. We finally recall the series identities
\begin{align}
     &\bleu{\HSeries(\x)}:=\bleu{\sum_{n\geq 0} h_n(\x)}= \bleu{\exp\big(\textstyle \sum_{k\geq 1} p_k(\x)/k\big)},\qquad {\rm and}\label{defn_Omega}\\
     &\bleu{\mathbb{E}(\x)}:=\bleu{\sum_{n\geq 0} e_n(\x)}= \bleu{\exp\big(\textstyle \sum_{k\geq 1} (-1)^{k-1} p_k(\x)/k\big)},
\end{align}  
In plethystic notation, the classical \define{Cauchy kernel identity} takes the form
     \begin{equation}\label{cauchy}
         \bleu{\HSeries[\x\y]= \prod_{i,j}\frac{1}{1-x_i\,y_j}=\sum_{\lambda} f_\lambda(\x)\, g_\lambda(\y)},
      \end{equation}
where $\{f_\lambda\}_\lambda$ and  $\{g_\lambda\}_\lambda$ are any \define{dual pair} of the bases for the \define{Hall scalar product}. 
Thus,
\begin{equation}\label{Classical_Cauchy}
     \bleu{h_n[\x\y]}=\bleu{\sum_{\lambda\vdash n}f_\lambda(\x)g_\lambda(\y)} \qquad {\rm iff}\qquad
 \bleu{\langle f_\lambda ,g_\mu\rangle} = \begin{cases}
     \bleu{1} & \text{if }\ \bleu{\lambda=\mu}, \\
     \bleu{0} & \text{otherwise}.
\end{cases}
      \end{equation}
Recall that the Hall-scalar product is characterized by the fact that the two bases $\{p_\lambda\}_\lambda$ and $\{{p_\lambda}/{z_\lambda}\}_\lambda$ constitute such a dual pair. Other dual pairs include:
the monomial basis $\{m_\lambda\}_\lambda$ together with the complete homogeneous basis $\{h_\lambda\}_\lambda$; and
the Schur functions $s_\lambda=s_\lambda(\x)$ is self-dual. Recall  that $s_\lambda[k]=s_\lambda(1,1,\ldots,1)$ may easily be calculated using the formula
\begin{equation}\label{eval_schur}
    \bleu{s_\lambda(\underbrace{1,\ldots,1}_{{k\ {\rm copies}}})=   \prod_{(i,j)\in\lambda}\frac{{k}-j+i}{\hook_\lambda{(i,j)}}},
\end{equation}
where $\hook_\lambda(c)$ stands for the hook length (see \autoref{sec_part}).
In particular,  $h_n[k]= \binom{k+n-1}{n}$, and $e_n[k]=\binom{k}{n}$.
We also have the following $q$-analog of the above identity
  \begin{equation}\label{qeval_schur}
    \bleu{s_\lambda(1,q,\ldots,q^{{k}-1})=  q^{\eta(\lambda)} \prod_{(i,j)\in\lambda}\frac{1-q^{{k}-j+i}}{1-q^{\hook_\lambda{(i,j)}}}},
\end{equation}
and in particular $h_n\big[\frac{1-q^k}{1-q}\big]=\qbinom{k+n-1}{n}_q$ and $e_n\big[\frac{1-q^k}{1-q}\big]=q^{\binom{n}{2}}\qbinom{k}{n}_q$.

%%%%%%%%%%%%%%%%%%% 
\subsection*{Adjoints and multiplication} 
Relative to the Hall scalar product, we consider the \define{adjoint} $f^\perp$ of the \define{multiplication}\footnote{Borrowing notations from~\cite{AnyLine}.}  operator $f^\bullet$ by a given symmetric function $f$. In formula, $f^\bullet(g):=f\cdot g$,
and  
	\begin{equation} \bleu{\langle f^\bullet  g,h\rangle = \langle g,f^\perp h\rangle},
	\end{equation}
for all symmetric functions $h$ and $g$.
Evidently, both $f\mapsto f^\bullet$ and $f\mapsto f^\perp$ are ring homomorphism from $\Lambda_{\x}$ to $\End(\Lambda_{\x})$: 
\begin{align}
   &\bleu{1^\bullet} = \bleu{\Id}, &&\bleu{(f+g)^\bullet} = \bleu{f^\bullet+g^\bullet}, && \bleu{(f\cdot g)^\bullet} = \bleu{f^\bullet\circ g^\bullet},\\
   &\bleu{1^\perp} =\bleu{ \Id}, &&\bleu{(f+g)^\perp} = \bleu{f^\perp+g^\perp}, && \bleu{(f\cdot g)^\perp} =\bleu{ f^\perp\circ g^\perp}.
\end{align}
It is worth recalling the two classical \define{Pieri formulas}:
\begin{equation}
   \bleu{h_k^\bullet (s_\mu)(\x) }= \bleu{\sum_{\lambda/\mu} s_\lambda(\x)}\qquad {\rm and}\qquad
    \bleu{h_k^\perp (s_\lambda(\x))} =\bleu{ \sum_{\lambda/\mu} s_\mu(\x)},
\end{equation}
in which indices run over skew partitions that are horizontal bands having $k$ cells.	More generally, the above link between the operators  $s_\mu^\perp$ and  $s_\mu^\bullet$ makes it clear why that the same \define{Littlewood-Richardson} coefficients $c_{\mu\nu}^\lambda$ occur in the identities
\begin{align}
   \bleu{s_\mu^\bullet s_\nu(\x)}&= \bleu{s_\mu(\x) s_\nu(\x)} \nonumber\\
   &=\bleu{\sum_{\lambda} c_{\mu\nu}^\lambda \,s_\lambda(\x)},\\
      \bleu{s_\mu^\perp s_\lambda(\x)}&= \bleu{\sum_{\nu} c_{\mu\nu}^\lambda \,s_\nu(\x)}.
\end{align}
For sure, $s_\mu^\perp s_\lambda$ coincides with the usual \define{skew-Schur} function $s_{\lambda/\mu}$.

%%%%%%%%%%%
\subsection{Expansions in the $\pi_\mu$ basis} For calculations in the elliptic Hall algebra, we need to be able to expand symmetric functions in the $\pi_\mu$ basis. Hence we need formulas for the $c_\mu(q,t)$ such that
 \begin{equation}\label{f_coeff_pi_mu}
       \bleu{f(\x)}=\bleu{\sum_{\mu\vdash d} c_\mu(q,t) \, \pi_\mu(\x)}.
    \end{equation}
Recall that
    \begin{equation}
         \bleu{\pi_n(\x)}:=\bleu{\sum_{a+\ell=n-1} (-qt)^{-a} s_{(a\,|\,\ell)}(\x)},\qquad {\rm and}\qquad
         \bleu{\pi_\mu(\x)}:=\bleu{\pi_{\mu_1}(\x)\pi_{\mu_2}(\x)\cdots \pi_{\mu_k}(\x)}.
    \end{equation}
 Observe that at $t=1/q$, the symmetric function $\pi_n(\x)$ specializes to the power-sums $(-1)^{n-1} p_n(\x)$, hence the $\pi_\mu$'s are linearly independent.  The symmetric functions $\pi_n$ may be expressed as
\begin{equation}
        \bleu{\pi_{n}} =\bleu{\frac{h_n[(1-qt)\,\x]}{e_n[1-qt]}}
    \end{equation}
hence, for all partition $\mu$ of $n$, we get
     \begin{equation}
     \bleu{\pi_{\mu}(\x)}=\bleu{\frac{h_\mu[(1-qt)\,\x]}{e_\mu[1-qt]}}.
  \end{equation} 
Moreover, since
     \begin{align}
     \bleu{h_n[\x\y]}
       &=\bleu{\sum_{\mu\vdash n} m_\mu[{\y}/{(1-qt)}]\, h_\mu[\x(1-qt)]},
  \end{align} 
the dual basis of the $\pi_\mu$'s is:
     \begin{equation}
     \bleu{\rho_{\mu}(\x)}=\bleu{e_\mu[1-qt]\, m_\mu[{\x}/{(1-qt)}]}.
       \end{equation} 
Using the above identities, we deduce the following.
\begin{lemma}
For any $n$,
   \begin{align}
        \bleu{h_n(\x)} &= \bleu{\sum_{\mu\vdash n} e_\mu[1-qt]\, m_\mu[{1}/{(1-qt)}]\, \pi_{\mu}(\x)},\quad {\rm and}\label{formule_h_pi}\\
        \bleu{e_n(\x)} &= \bleu{\sum_{\mu\vdash n} e_\mu[1-qt]\, f_\mu[{1}/{(1-qt)}]\,\pi_{\mu}(\x)};\label{formule_e_pi}
   \end{align}
\end{lemma}
where $f_\mu$ denotes the ``forgotten'' symmetric functions (dual to $e_\mu$).

%%%%%%%%%%%%%%%%%%%%%%%%%%%%%%%
 \subsection{Composition indexed Schur function} 
To simplify some of our upcoming expressions, it is useful to allow Schur functions to be indexed by composition, with the understanding that the resulting expression should be decoded using \autoref{formule_jacobi}. Thus for a composition $\alpha=\alpha_1\alpha_2\cdots \alpha_k$, one defines
    \begin{equation}\label{jacobi2}
       \bleu{s_\alpha:=\det\left( h_{\alpha_i+(j-i)}\right)_{1\leq i,j\leq k}}.
    \end{equation}
Exchanging rows $i$ and $i+1$, we see that
 	\begin{equation}
	    \bleu{s_\alpha=-s_\beta}, \qquad {\rm whenever}\qquad \bleu{\beta = \alpha_1\alpha_2\cdots (\alpha_{i+1}-1)(\alpha_i+1)\cdots \alpha_k}.
	\end{equation}
In particular, we get that $s_\alpha=0$, if $\alpha_{i+1}=\alpha_{i}+1$ for some $i$. With repeated application of the above rule, one may ``straightened'' any composition indexed Schur $s_\alpha$, to either get $0$ or a signed partition indexed Schur function $\pm s_\lambda$. For example, considering the compositions of $6$ that are not already partitions, we have
\begin{align*} 
    &s_{114} =s_{222}, && s_{132} =s_{213}=-s_{222} && s_{15}=-s_{42}\\
    &s_{24}=-s_{33}, && s_{141} = -s_{321},&& s_{1311}=-s_{2211},
\end{align*}
with $0$ as value for all other cases. 

 %%%%%%%%%%%%%%%%%%%%%%%%%%%%%%%
%%%%%%%%%%%%%%%%%%%%%%%%%%%%%%%

   \section{Operators  \texorpdfstring{$\Dk{k}$}{DD}}
It is shown in \cite{nabla} that the operators\footnote{In which they are denoted $D_k$, and we consider $k\in \mathbb{Z}$.} $\Dk{k}:=\mathcal{X}^{(1,k)}$ may be directly
defined by the formal power series identity 
        \begin{equation}\label{defDn}
          \bleu{\mathcal{X}(z)}=   \bleu{\sum_{k=-\infty}^{\infty} \Dk{k}\,z^k:=  \HSeries[-z\x]^\bullet\, \HSeries\left[\x M/{z}\right] ^\perp}.
      \end{equation}   
 Equivalently, since $\HSeries[\y]^\perp f(\x)= f\left[\x+\y\right]$, we have
$\mathcal{X}(z)(f(\x))=  \HSeries[-z\x]^\bullet f\left[\x+M/{z}\right]$.
For our calculations in the elliptic Hall algebra, it is thus interesting to show the following.
\begin{lemma}\label{Dk_reducesto_D0}
For all $k$ and $j$ we have $\Dk{k+j}= [\Dk{k},p_j[\x/M]^\bullet]$, and 
$\Dk{k-j}= (-1)^j [p_j^\perp,\Dk{k}]$.
 \end{lemma}
 \begin{proof} 
The first identity is equivalent to 
 ${z^j [\mathcal{X}(z),p_j(\x/M)^\bullet]=  \mathcal{X}(z)}$ which we may check via the following calculation.
Observe that  $p_j[(\x+M/z)/M]^\bullet -p_j[\x/M]^\bullet  = z^{-j}$.
Hence:
\begin{eqnarray*}
        z^j\,  \Lie{\mathcal{X}(z)}{p_j[\x/M]^\bullet}\, f(\x) 
        	&=&z^j\sum_{k=-\infty}^{\infty} \Lie{\Dk{k}}{p_j[\x/M]^\bullet}(f(\x))\, z^{k}\\
        	&=&z^j\sum_{k=-\infty}^{\infty} \Dk{k} \left(p_j[\x/M] f(\x)\right)z^{k} -
	            z^j p_j[\x/M]\sum_{k=-\infty}^{\infty}  \Dk{k} f(\x) z^{k}\\
        	 &=&z^j \HSeries[-z\x]^\bullet p_j \!\left[\x +M/{z}\right]f \!\left[\x +M/{z}\right] \\
          &&\qquad\qquad\qquad -  z^j \HSeries[-z\x]^\bullet p_j(\x)  f \!\left[\x +M/{z}\right]  \\                   
         &=&z^j \HSeries[-z\x]^\bullet \Big(z^{-j} \,f \!\left[\x +M/{z}\right]   \Big)= \mathcal{X}(z) \, f(\x).                
  \end{eqnarray*}
  In particular, we get $ \Dk{k}= [\Dk{0},p_k[\x/M]^\bullet]$. 
  
The second identity (essentially following the calculation in the proof of~\cite[Lemma 3.3.3]{AnyLine}) is obtained similarly. Observe first that
  \begin{equation}
          \Lie{p_j^\perp}{\HSeries[-z\x]^\bullet \HSeries\left[\x M/{z}\right] ^\perp}
      	=\Lie{p_j^\perp}{\HSeries[-z\x]}\  \HSeries\left[\x M/{z}\right] ^\perp 
\end{equation}
since $f^\perp$ operators commute. Observe that $\HSeries[-z\x] = \exp(\sum_{j\geq 0} p_i (-z)^i/i)$, and that $p_j^\perp$ acts almost as a derivation on expression in $p_i$ \footnote{In other words, they generate a Heisenberg algebra with commutation relation $ [p_i,p_j^\perp]= i \delta_{i,j}$.}, with  
   $$p_j^\perp p_i/i =  \begin{cases}
      1 & \text{if}\ i=j, \\
      0 & \text{otherwise}.
\end{cases}$$
Thus we get that
   $$ p_j^\perp \HSeries[-z\x]^\bullet = (-z)^j\HSeries[-z\x]^\bullet+\HSeries[-z\x]^\bullet p_j^\perp. $$
\begin{eqnarray*}
         [p_j^\perp,\mathcal{X}(z)]
        	&=& \left[p_j^\perp\,,\, \HSeries[-z\x]^\bullet\, \HSeries\left[\x M/{z}\right] ^\perp\right]\\
      	&=& \Big[p_j^\perp\,,\, \HSeries[-z\x]^\bullet \Big]\ \HSeries\left[\x M/{z}\right] ^\perp \\
       	&=& (-z)^j\,\HSeries[-z\x]^\bullet\, \HSeries\left[\x M/{z}\right] ^\perp = (-z)^j\,\mathcal{X}(z).
  \end{eqnarray*}
In particular, $\Dk{-j}=(-1)^j [p_j^\perp,\Dk{0}]$.
  \end{proof}
 In simple words, all the operators $\Dk{k}$ may be obtained from $\Dk{0}$.

%%%%%%%%%%%%%%%%%%%%%%%%%%%%%%%%%%%%%%%%%%%%
%%%%%%%%%%%%%%%%%%%%%%%%%%%%%%%%%%%%%%%%%%%%
\renewcommand\refname{References and Further Readings}

\end{document}